\documentclass[12pt, leqno]{article}
\usepackage[utf8]{inputenc}
\usepackage{setup_normal2}
\usepackage{newcommands}
\usepackage{titletoc}
\usepackage[most]{tcolorbox}
\begin{document}
\newtheorem{theoreme}{Theorem}
\newtheorem{ex}{Example}
\newtheorem{lemme}{Lemma}
\newtheorem{remarque}{Remark}
\newtheorem{exemple}{Example}
\newtheorem{corolaire}{Corollary}
\newtheorem*{Conjecture}{Conjecture}
\newtheorem{hyp}{Hypothesis}
\newcommand\Sp{\mathbb{S}}
\providecommand{\keywords}[1]{\textbf{\textit{Index terms---}} #1}

\title{Local Weak Limits of Laplace Eigenfunctions}
\author{Maxime Ingremeau}



\maketitle

\begin{abstract}
In this paper, we introduce a new notion of convergence for the Laplace eigenfunctions in the semiclassical limit, the local weak convergence. This allows us to give a rigorous statement of Berry's random wave conjecture. Using recent results of Bourgain, Buckley and Wigman, we will prove that some deterministic families of eigenfunctions on $\T^2$ satisfy the conclusions of the random wave conjecture. We also show that on an arbitrary domain, a sequence of Laplace eigenfunctions always admits local weak limits. We explain why these local weak limits can be a powerful tool to study the asymptotic number of nodal domains.
\end{abstract}

\section{Introduction}
In his seminal paper \cite{berry1977regular}, Berry  suggested that high-frequency eigenfunctions of the Laplacian in geometries where the classical dynamics is sufficiently chaotic (for instance, negatively curved manifolds) should behave like random combinations of plane waves. This heuristics, known as the random wave model (RWM), has led to many conjectures concerning the $L^p$ norms, semiclassical measures or nodal domains of chaotic eigenfunctions. Several of these conjectures have been checked numerically (\cite{hejhal1992topography}, \cite{aurich1993statistical}, \cite{backer1998rate}, \cite{barnett2006asymptotic}) or experimentally (\cite{savytskyy2004experimental}, \cite{hul2005investigation} \cite{kuhl2007nodal}). However, it is not clear how Berry's general idea should be formulated in a rigorous way: saying that a sequence of deterministic objects behave asymptotically in a random way can be interpreted in different ways. The reader can for instance refer to \cite{rudnick1994behaviour}, \cite{zelditch2010recent} and \cite{nonnenmacher2013anatomy} for different mathematical perspectives on Berry's conjecture.

In this paper, we introduce another interpretation of Berry's random waves conjecture by associating to a sequence of Laplace eigenfunctions a sequence of measures on an abstract Polish space, which we call \emph{local measures}. We show that we may always extract a subsequence of local measures which will converge. The limit, which we name a \emph{local weak limit} of the sequence of eigenfunctions, is a measure on the space
\begin{equation}\label{eq:DefFP}
FP := \{ f\in C^\infty(\R^d; \R) \text{ such that } -\Delta f = f\},
\end{equation}
whose topology is given by the distance
\begin{equation}\label{eq:DefDist}
d(f,g):= \inf \{\varepsilon>0; \sup_{|x|< \varepsilon^{-1}} |f(x)-g(x)| <\varepsilon \}.
\end{equation}

More precisely, let $\Omega$ be an open set in $\R^d$, or $\Omega=\T^d$, and $\phi_n$ be an orthogonal sequence of real-valued eigenfunctions of the Dirichlet Laplacian on $\Omega$, satisfying\footnote{The reason for this normalization, which we will take in all the paper, is the following: we will often take a point $x_0$ uniformly at random in $\Omega$, and we want $|\phi_n(x_0)|^2$ to have average value 1.}  $\|\phi_n\|_{L^2(\Omega)}^2= \Vol(\Omega)$ and
\begin{equation*}
-h_n^2 \Delta \phi_n = \phi_n.
\end{equation*}
 For any Borel set $U\subset \Omega$  of positive Lebesgue measure, we will define in Section \ref{sec:Construction} local measures $LM_U(\phi_n)$ associated to $\phi_n$, which are Borel measures on the set $C(\R^d)$ equipped with the distance defined in (\ref{eq:DefDist}). We will then define $\sigma_U\left((\phi_n)_n\right)$ to be the set of accumulation points of $\left(LM_U(\phi_n)\right)_n$ for weak convergence. We will show that for any sequence $(\phi_n)$ of eigenfunctions, $\sigma_U((\phi_n)_n)$ is never empty, and it is supported on $FP$.

Local measures and local weak limits are quite technical to introduce, and we defer their precise definition to the next section. However, the idea behind is rather simple. A solution to the equation $-h_n^2\Delta \phi_n = \phi_n$, when rescaled to a ball $B(x_0, R h_n)$, will resemble\footnote{In general, it will not be equal to an element of $FP$. Indeed, if this were the case, the solution could be continued as an analytic function on all of $\R^d$.} an element of $FP$. The local measure associated to $\phi_n$ will somehow ``count how many times we will resemble a given element of $FP$ when varying the point $x_0$ in $U$".
The notion of local weak convergence we introduce here was inspired by local weak convergence of graphs, also known as Benjamini-Schramm convergence (\cite{benjamini2011recurrence}).

Local weak convergence of eigenfunctions allows us to give a rigorous statement of Berry's conjecture about Laplace eigenfunctions in chaotic billiards. We refer the reader to 
\cite{chernov2006chaotic} for the definition and examples of chaotic billiards. Our approach could also be used to formulate Berry's conjecture on manifolds whose geodesic flow is Anosov (for instance, on manifolds of negative curvature). However, on manifolds, local weak limits will in general depend on a choice of a local orthonormal frame; these issues will be pursued elsewhere. Note that, on random regular graphs, a weak version of Berry's conjecture was proven in \cite{backhausz2019almost}.

\subsubsection*{Random Gaussian Fields as Local Weak Limits and Berry's conjecture}
The isotropic monochromatic Gaussian random field $\Psi_{Berry}: \R^d\rightarrow \R$ is uniquely defined as the centred stationary Gaussian random field, with covariance function
$$\E [\Psi_{Berry}(x) \Psi_{Berry}(x')]= \int_{\Sp^{d-1}} e^{i(x-x')\cdot \theta} \mathrm{d}\theta.$$
We refer the reader to \cite{abrahamsen1997review} for more details on Gaussian random fields.

In dimension 2, $\Psi_{Berry}$ can alternatively be defined, in polar coordinates, as
\begin{equation*}
\Psi_{Berry}(r,\theta) = X_0 J_{0} (kr) + \sqrt{2} \sum_{n\geq 1} J_{n} (kr) \left[X_n \cos(n\theta) + Y_n \sin (n\theta) \right],
\end{equation*}
 where $J_n$ is the $n$-th Bessel function, and where the $(X_n)_{n\geq 0}, (Y_n)_{n\geq 1}$ are independent families of standard Gaussian variables (see \cite[\S 4.2]{nazarov2010random}).

Almost surely, $\Psi_{Berry}$ is an element of $FP$, so that, if $A\subset FP$ is a Borel set, $\mathbb{P}(\Psi_{Berry} \in A)$ is well-defined, and
$$\mu_{Berry}: A \mapsto \mathbb{P}(\Psi_{Berry}\in A)$$
defines a measure on $FP$.

\begin{Conjecture}[Berry's Random Wave Conjecture]
Let $\Omega\subset \R^d$ be a \emph{chaotic} billiard, and let $(\phi_n)$ be an orthogonal sequence of real-valued eigenfunctions of the Dirichlet Laplacian in $\Omega$, satisfying $\|\phi_n\|_{L^2(\Omega)}^2=\Vol(\Omega)$. Then
$$\sigma_\Omega((\phi_n)_n) = \{\mu_{Berry}\}.$$
\end{Conjecture}

An analogous statement of Berry's conjecture on manifolds of negative curvature was given in \cite{abert2018eigenfunctions}.

In this paper, the authors consider the \emph{level aspect} framework, which is more general than semiclassical asymptotics: they consider a sequence of manifolds, and sequences of eigenfunctions on these manifolds, whose eigenvalues lie in a fixed interval. For instance, starting from a given Riemannian manifold, we can rescale it, and recover the semiclassical asymptotics; but we could also consider sequences of hyperbolic surfaces with a topology of increasing complexity, and study the asymptotics of a sequence of eigenfunctions on these manifolds.
The authors state a version of Berry's conjecture in the level aspect, when the manifolds considered are locally symmetric, and form an expanding sequence.

Although the framework considered in \cite{abert2018eigenfunctions} is rather different from ours (they mainly consider eigenfunctions on sequences of locally symmetric spaces while we consider sequences of eigenfunctions on a domain in $\R^d$), we insist that their Conjecture 1 is essentially the same as our version of Berry's conjecture, and we find their point of view very interesting, and complementary to ours. 

\subsubsection*{Consequences of Berry's conjecture}
In Section \ref{sec:Consequences}, we will state several consequences of our interpretation of Berry's conjecture. If $(\phi_n)$ is a sequence of eigenfunctions such that $\sigma_\Omega((\phi_n)_n) = \{\mu_{Berry}\}$, then
\begin{itemize}
\item $\phi_n$ satisfies quantum unique ergodicity (see Section \ref{sec:QUE}).
\item The number of nodal domains of $\phi_n$ grows at least as $c_{NS} h_n^{-d}$, where $c_{NS}$ is the Bogomolny-Schmit constant, or Nazarov-Sodin constant, which was introduced in \cite{bogomolny2002percolation} and \cite{nazarov2009number} (see section \ref{sec:Nodal}). One would expect that the number of nodal domains is actually equivalent to $c_{NS} h_n^{-d}$. We could only show it in dimension 2, for analytic domains\footnote{Actually, what we really need is that the upper bound in Yau's conjecture on the length of the nodal set holds. This is known for analytic manifolds thanks to \cite{DF}, but the conjecture is still open on smooth surfaces, in spite of the progresses made in \cite{DFsurfaces}, and more recently in \cite{logunovSurf}.}. We believe that local weak limits are a powerful tool to study lower bounds on the number of nodal domains of eigenfunctions, even in the situations where the limit measure is not $\mu_{Berry}$: see the discussion at the end of \ref{sec:Nodal}.
\item We have $\|\phi_n\|_{L^\infty} \underset{n\to +\infty}{\longrightarrow} +\infty$. Note that  our interpretation of Berry's conjecture gives no interesting upper bound on the $L^\infty$ norm of eigenfunctions. Indeed, we can find sequences of eigenfunctions $(\phi_n)$ satisfying our version of Berry's conjecture,  but such that $\|\phi_n\|_{L^\infty} \geq c h_n^{-\frac{d-1}{2}}$ for some $c>0$. In other words, our version of Berry's conjecture does not allow us to improve the $L^\infty$ bound on eigenfunctions, as any normalized sequence of eigenfunctions $(\phi_n)$ satisfies the Hörmander bound $\|\phi_n\|_{L^\infty} \leq C h_n^{-\frac{d-1}{2}}$.

The reason why our version of Berry's conjecture gives no interesting upper (or lower) bound on $L^\infty$ norms is that local weak limits capture how eigenfunctions look like on typical sets (i.e., on sets of large measure). But the places where the eigenfunctions are very large are not typical at all, so they disappear in the limit.
\end{itemize}

Other interpretations of Berry's conjecture led to upper bounds on $\|\phi_n\|_{L^\infty}$ which were logarithmic in $h_n$ (see for instance \cite[\S 4]{nonnenmacher2013anatomy} and the references therein). However, we know since \cite{rudnick1994behaviour} that on some compact hyperbolic manifolds of dimension 3, $\|\phi_n\|_{L^\infty}$ can grow polynomially with $h_n$. This has led the authors of \cite{rudnick1994behaviour} to say that the eigenfunctions on such manifolds did not satisfy Berry's conjecture; one could hence think that Berry's conjecture should only hold on \emph{generic} manifolds of negative curvature.
However, since our interpretation of Berry's conjecture does not contradict the results of \cite{rudnick1994behaviour}, we believe it should hold on any manifold of negative curvature (or, more generally, on manifolds whose geodesic flow is Anosov), and in chaotic billiards (in the sense of \cite{chernov2006chaotic}).

\subsubsection*{Random wave model for deterministic toral eigenfunctions}
The methods introduced by Bourgain, Buckley and Wigman in  \cite{bourgain2014toral} and \cite{buckley2015number} (see also \cite{sartori2020planck}) to study the number of nodal domains allow to prove that certain deterministic families of eigenfunctions on $\T^2$ satisfy the conclusion of Berry's conjecture (although no chaotic dynamics in present here).

On $\T^2= \R^2\backslash \Z^2$, the eigenvalues of the Laplacian are the numbers $(4\pi^2 E_n)_{n\in \N}$, where $E_n$ is the increasing sequence of numbers such that
$$\mathcal{E}_{E_n} := \{\xi\in \Z^2; |\xi|^2 = E_n\}$$
is non-empty. For such an $E_n$, a very specific associated eigenfunction is given by
\begin{equation}\label{eq:EigenToral1}
\varphi_n(x):= \frac{1}{|\mathcal{E}_{E_n}|^{1/2}} \sum_{\xi\in \mathcal{E}_n} e^{2i\pi x\cdot \xi}.
\end{equation}
\begin{theoreme}\label{th:Tore1}
There exists a density 1 sequence $n_j$ such that  we have
$$\sigma_{\T^2}((\varphi_{n_j})_j) = \{\mu_{Berry}\}.$$
\end{theoreme}
Actually, this theorem holds for eigenfunctions which are much more general than (\ref{eq:EigenToral1}). The precise assumptions we need are given in Hypothesis \ref{hyp:tore}, in Section \ref{sec:Torus}. Note that the methods of \cite{bourgain2014toral} and \cite{buckley2015number} are only valid in dimension 2, and we don't know if the statement remains true in higher dimensions.

\subsubsection*{Organisation of the paper}
In Section \ref{sec:Construction}, we give the definition of local measures and local weak limits of eigenfunctions. We prove that a sequence of local measures of eigenfunctions always has a subsequence which converges to a local weak limit. We then give some elementary examples of computations of local weak limits. In Section \ref{sec:Criteria}, we give some criteria to identify the local weak limits of a sequence of eigenfunctions. In Section \ref{sec:Torus}, we prove a more precise version of Theorem \ref{th:Tore1}. In Section \ref{sec:Consequences}, we explain several consequences of our interpretation of Berry's conjecture. Finally, in Appendix \ref{sec:App}, we recall the functionnal analytic properties of the spaces of functions we use in our constructions, based on distances similar to the one in (\ref{eq:DefDist}).

\subsubsection*{Acknowledgement}
The author would like to thank N. Anantharaman, N. Bergeron, A. Deleporte, E. Le Masson, S. Nonnenmacher and A. Rivera for their interest in our work and for useful discussion. 

He would especially like to thank the anonymous referees for their very careful reading and many valuable comments, which improved greatly the presentation of the paper.

The author was partially funded by the Labex IRMIA, and partially supported by the Agence Nationale de la Recherche project GeRaSic
(ANR-13-BS01-0007-01).

\section{Construction of local weak limits}\label{sec:Construction}
In this section, we will use the notations and results of Appendix \ref{sec:App}. In particular, for $k\in \N$, $\mathcal{H}^k$ is the space of $C^k$ functions equipped with a ``local" distance similar to  the one introduced in (\ref{eq:DefDist}), $\mathcal{M}^k$ is the space of finite signed measures on $\mathcal{H}^k$, and $\mathcal{C}_b(\mathcal{H}^k)$ is the space of \emph{bounded} real-valued continuous functions on the metric space $\mathcal{H}^k$.

\subsection{Construction of local measures}
In the sequel, we will fix $\Omega\subset \R^d$ a bounded open set,
and we consider a sequence $\phi_n\in C^\infty(\Omega ; \R)$, 
 and a sequence $h_n>0$ going to zero.

We will always suppose that we have, for some $K\in \N$
\begin{equation}\label{eq:CondSobolev}
\|\phi_n\|_{H^K}\leq C_K h_n^{-K}.
\end{equation}

In particular, we will often consider the special case where $\|\phi_n\|^2_{L^2}=\Vol(\Omega)$, and 
\begin{equation}\label{eq:DefEigen}
\begin{cases}
-h_n^2 \Delta \phi_n &= \phi_n\\
\phi_n|_{\partial \Omega} &= 0,
\end{cases}
\end{equation}
which satisfies (\ref{eq:CondSobolev}) for all $K\in \N$.

Let us fix $\chi \in C_c^\infty([0,\infty); [0,1])$ a decreasing function taking value one in a neighbourhood of the origin and vanishing outside $[0,1]$.

\begin{definition}\label{def:Local}
If $x_0\in \Omega$ and $n\in \N$, we define a function $\tilde{\phi}_{x_0,n}\in C_c^\infty(\R^d)$ by
\begin{equation*}
\tilde{\phi}_{x_0,n} (y) :=  \phi_n \big{(} x_0 +h_n y\big{)} \chi \Big{(} \frac{h_n |y|}{d (x_0,\partial \Omega)}\Big{)}.
\end{equation*}
\end{definition} 

\begin{remarque}\label{rem:Torus}
All the constructions and all the proofs of this section also work when $(\phi_n)$ is a sequence of functions on $\T^d$ satisfying (\ref{eq:CondSobolev}), and hence, for eigenfunctions on the torus, by defining the functions $\tilde{\phi}_{x_0,n}\in C_c^\infty(\R^d)$ by
\begin{equation*}
\tilde{\phi}_{x_0,n} (y) :=  \phi_n \big{(} x_0 +h_n y\big{)} \chi(h_n|y|).
\end{equation*}
The proof of Lemma \ref{lem:Tight} is slightly simpler in this case, since we don't have to take the boundary of $\Omega$ into account.
\end{remarque}
 
For each $k\in \N$, $n\in \N$, for each $x_0\in \Omega$, we have $\tilde{\phi}_{x_0,n}\in \mathcal{H}^k$, so that we may define $\delta_{\tilde{\phi}_{x_0,n}} \in \mathcal{M}^k$.
For each $n\in \N$, $k\in \N$ and each Borel set of positive Lebesgue measure $U\subset \Omega$, we then define the \emph{$\mathcal{H}^k$-local measure of $\phi_n$ on $U$} as
\begin{equation}\label{eq:DefLM}
LM_{k,U}(\phi_n) := \frac{1}{\Vol(U)} \int_U \mathrm{d}x_0 \delta_{\tilde{\phi}_{x_0,n}}.
\end{equation} 
This defines a probability measure in $\mathcal{M}^k$.

Taking $\iota_k$ as in (\ref{eq:ExtensionProp}), we also define $LM^{\iota_k}_{k,U}(\phi_n)\in \big{(}\mathcal{C}_b(FP)\big{)}^*$ by
\begin{equation*}
\forall F\in \mathcal{C}_b(FP),~~  \langle LM_{k,U}^{\iota_k}(\phi_n), F\rangle =  \langle LM_{k,U}(\phi_n), \iota_k F\rangle.
\end{equation*}

\subsection{Definition and properties of local weak limits}\label{subsec:DefPropLWL}
\begin{lemme}\label{lem:Tight}
Let $\Omega\subset \R^d$ be a smooth bounded open set, and let $k\in \N$. Let $K_0=K_0(k,d)$ be the smallest integer such that $H^{K_0}(\R^d)$ can be continuously embedded in $C^{k+1}(\R^d)$.

Suppose that $(\phi_n)$ and $(h_n)$ satisfy (\ref{eq:CondSobolev}) for some $K\geq K_0$. Then the sequence $\left(LM_{k,U}(\phi_n)\right)_n$ is tight in $\mathcal{M}^k$, i.e., $\forall \varepsilon>0$, there exists a compact set $K_\varepsilon\subset \mathcal{H}^k$ such that $LM_{k,U}(\phi_n) (\mathcal{H}^k \backslash K_\varepsilon) <\varepsilon$ for all $n\in \N$.
\end{lemme}
\begin{proof}
Let $\varepsilon>0$, and let us write $\Omega_\varepsilon := \{x\in \Omega; \mathrm{d}(x,\partial \Omega)>\varepsilon\}$. We have $\mathrm{Vol}(\Omega\backslash \Omega_\varepsilon) = o_{\varepsilon\rightarrow 0}(1)$.

Recalling that $h_n$ goes to zero, and is hence bounded, for any multiindex $\beta \in \N^d$ and any $\varepsilon>0$, we may find $C_{\beta,\varepsilon}$ such that for all $x\in \Omega_\varepsilon$, all $y\in \R^d$ and all $n\in \N$, we have
\begin{equation}\label{eq:BoundDerivatives}
D^\beta_y \left[ \chi \Big{(} \frac{h_n |y|}{d (x,\partial \Omega)}\Big{)}\right]\leq C_{\beta,\varepsilon} .
\end{equation}
We deduce that for any $\ell >0$,

\begin{equation*}
\begin{aligned}
\int_{\Omega_\varepsilon} \mathrm{d}x \|\tilde{\phi}_{x,n}\|_{H^{K}(B(0,\ell))}^2
&= \sum_{|\alpha|\leq K} \int_{\Omega_\varepsilon}  \int_{B(0,\ell)} \left| D^\alpha_y \left[ \chi \Big{(} \frac{h_n |y|}{d (x,\partial \Omega)}\Big{)} \phi_n(x + h_n y) \right]\right|^2 \mathrm{d}y \mathrm{d}x \\
&\leq C_{\varepsilon,K} \sum_{|\beta|\leq K} \int_{\Omega_\varepsilon}  \int_{B(0,\ell)} h_n^{2|\beta|} \left| \left(D^\beta \phi_n\right)(x + h_n y) \right|^2  \mathrm{d}y \mathrm{d}x ~~\text{ thanks to (\ref{eq:BoundDerivatives})}\\
&\leq C_\varepsilon \mathrm{Vol}(B(0,\ell)) \sum_{|\beta|\leq K} h_n^{2|\beta|} \int_\Omega |D^\beta \phi_n (z)|^2 \mathrm{d}z\\
 &\leq C_\varepsilon' \ell^d,
\end{aligned}
\end{equation*}
thanks to (\ref{eq:CondSobolev}).
To go from the second line to the third, we used Fubini's theorem, and a change of variables.

Using the assumption we made on $K$, we deduce that for any $\ell>0$, $\varepsilon>0$, we may find $C(\ell, \varepsilon)$ such that, for all $n\in \N$, we have
$$\int_{\Omega_\varepsilon} \mathrm{d}x \|\tilde{\phi}_{x,n}\|_{C^{k+1}(B(0,\ell))}^2\leq C(\ell, \varepsilon).$$

Set $a_\varepsilon(\ell):= \left(\varepsilon^{-1} 2^\ell C(\ell, \varepsilon)\right)^{1/2}$. By Markov's inequality, we deduce that for all $n,\ell \in \N$, there exists  $V_\varepsilon(n,\ell)\subset \Omega_\varepsilon$ with $\mathrm{Vol}(V_\varepsilon(n,\ell))<2^{1-\ell} \varepsilon$  for all $x\in \Omega_\varepsilon \backslash V_\varepsilon(n,\ell)$, we have 
$$\|\tilde{\phi}_{x,n}\|_{C^{k+1}(B(0,\ell))}\leq a_\varepsilon(\ell).$$

In particular, if we write $\boldsymbol{a}_\varepsilon = (a_\varepsilon(\ell))_\ell$, then for each $\varepsilon>0$ and each $n\in \N$, we have
\begin{equation}\label{eq:AlmostCompact}
\Vol\big{(} \{x\in \Omega \text{ such that } \tilde{\phi}_{x,n} \notin \mathcal{H}^{k+1} (\boldsymbol{a}_\varepsilon)\}\big{)}\leq \Vol\left(\Omega\backslash \Omega_\varepsilon\right)+\sum_\ell 2^{1-\ell} \varepsilon\leq \Vol(\Omega\backslash \Omega_\varepsilon)+ 4 \varepsilon.
\end{equation}

Therefore, we have, for each $n\in \N$ and each $\varepsilon>0$ that 
$$LM_{k,U}(\phi_n) \big{(}\mathcal{H}^k\backslash \mathcal{H}^{k+1}(\boldsymbol{a}) \big{)} < 4\varepsilon+ \Vol(\Omega\backslash \Omega_\varepsilon).$$
The statement then follows from  Lemma \ref{lem:CompactSubset}.
\end{proof}

As a consequence of Lemma \ref{lem:Tight} and of Prokhorov's theorem, we have
\begin{corolaire}\label{cor:Compactness}
Let $(\phi_n)$ satisfy (\ref{eq:CondSobolev}).
Let $k\in \N$, $U\subset \Omega$ be a Borel set of positive Lebesgue measure.
There exists a subsequence $n_j$ and a probability measure $\mu\in \mathcal{M}^k$ such that $LM_{k,U}(\phi_{n_j})\overset{\ast}{\rightharpoonup} \mu$, i.e.,
 for all $F \in \mathcal{C}_b(\mathcal{H}^k)$, we have
\begin{equation*}
\lim\limits_{n\rightarrow \infty} \langle LM_{k,U}(\phi_{n_j}), F \rangle =  \langle \mu,F\rangle.
\end{equation*}

In particular, there exists $\nu \in \big{(}\mathcal{C}_b(FP)\big{)}^*$ such that for all $F\in \mathcal{C}_b(FP)$, we have
\begin{equation*}
\lim\limits_{n\rightarrow \infty} \langle LM^{\iota_k}_{k,U}(\phi_{n_j}),F  \rangle =  \langle  \nu,F \rangle= \langle  \mu, \iota_k F\rangle .
\end{equation*}
\end{corolaire}

\begin{definition}
We will denote by $\sigma_{k,U}\left((\phi_n)_n\right)$ the set of accumulation points of $LM_{k,U}(\phi_n)$ for the weak-* topology, and by $\sigma^{\iota_k}_{k,U}(\phi_n)$ the set of accumulation points of $LM^{\iota_k}_{k,U}\left((\phi_n)_n\right)$ for the weak-* topology.
\end{definition}

From now on, we will only consider the case when $(\phi_n)$ satisfies (\ref{eq:DefEigen}).
We shall see in Corollary \ref{cor:Intrinsic} below that $\sigma^{\iota_k}_{k,U}((\phi_n)_n)\subset \mathcal{M}$, and that this set does not depend on $k$ and on $\iota_k$. First of all we will show that local weak limits of eigenfunctions are always supported on $FP$.

\begin{lemme}
Let $(\phi_n)$ satisfy (\ref{eq:DefEigen}).
Let $U\subset \Omega$, $k\in \N$ and let $\mu\in \sigma_{k,U}((\phi_n)_n)$. Then $\mu$ is supported on $FP$.
\end{lemme}
\begin{proof}
Since $(\phi_n)$ satisfies (\ref{eq:DefEigen}), we have
we have
$$\|(-\Delta-1)\tilde{\phi}_{n,x_0}\|_{C^k(B(0,r))}\leq C\left(r, \mathrm{d}(x_0,\partial \Omega)\right) h_n \|\tilde{\phi}_{n,x_0}\|_{C^{k+1}(B(0,r))}.$$

Thanks to (\ref{eq:AlmostCompact}), we deduce that, for any $\varepsilon>0$, we can find $C_\varepsilon$ such that
\begin{equation}\label{eq:SmallVolumeBadPoints}
\Vol\big{(} \{x_0\in \Omega \text{ such that } \boldsymbol{d}_k\big{(}(-\Delta-1)\tilde{\phi}_{n,x_0}, 0 \big{)} \geq C_\varepsilon h_n \}\big{)}\leq  \varepsilon.
\end{equation}

Let $A$ be a Borel set which does not intersect $FP$. By Ulam's Theorem (\cite[Theorem 7.1.4]{dud}), any measure on a Polish space is regular, so that 
\begin{equation}\label{eq:InnerReg}
\mu(A) = \sup \{ \mu(K); ~ K \text{ compact }, K\subset A\}.
\end{equation}
Let us take a set $K\subset A$, compact for the $\mathcal{H}^k$ topology.
For each $v\in K$, we have $\boldsymbol{d}_{k-2}(v, -\Delta v) >0$. $v\mapsto \boldsymbol{d}_{k-2}(v, -\Delta v)$ is continuous for the $\mathcal{H}^k$ topology, so by compactness, we may find $c>0$ such that for all $v\in K$, we have $\boldsymbol{d}_{k-2}(v, -\Delta v) \geq c$.

Suppose that $x_0\in \Omega$ is such that $ \boldsymbol{d}_{k-2}\big{(}(-\Delta-1)\tilde{\phi}_{n,x_0}, 0 \big{)}\leq \frac{c}{3}$. Then we have for any $v\in K$
\begin{equation*}
\begin{aligned}
\boldsymbol{d}_k (\tilde{\phi}_{x_0,n},v) &\geq \boldsymbol{d}_{k-2} \big{(} -\Delta \tilde{\phi}_{x_0,n},-\Delta v \big{)}\\
&\geq \boldsymbol{d}_{k-2} \big{(}\tilde{\phi}_{x_0,n},-\Delta v \big{)} - \frac{c}{3}.
\end{aligned}
\end{equation*}

Therefore, we must have $\boldsymbol{d}_k (\tilde{\phi}_{x_0,n},v) \geq \frac{c}{3}$, so that $\boldsymbol{d}_k (\tilde{\phi}_{x_0,n}, K) \geq \frac{c}{3}$.

Thus, we have just shown that for any $x_0\in \Omega$,
\begin{equation}
\left[\boldsymbol{d}_{k-2}\big{(}(-\Delta-1)\tilde{\phi}_{n,x_0}, 0 \big{)}\leq \frac{c}{3} \right] \Longrightarrow \boldsymbol{d}_k (\tilde{\phi}_{x_0,n}, K) \geq \frac{c}{3}.
\end{equation}

In particular,
$$\Vol\left( \{x_0\in \Omega \text{ such that } \boldsymbol{d}_k (\tilde{\phi}_{x_0,n}, K) \leq \frac{c}{3} \}\right) \leq \Vol\big{(} \{x_0\in \Omega \text{ such that } \boldsymbol{d}_k\big{(}(-\Delta-1)\tilde{\phi}_{n,x_0}, 0 \big{)} \geq  \frac{c}{3} \}\big{)},$$
which goes to zero as $n\to \infty$ thanks to  (\ref{eq:SmallVolumeBadPoints}).

Let $V:= \{f\in \mathcal{H}^k; \boldsymbol{d}_k(f, K)< \frac{c}{3}\}$, which is an open set. From what precedes, we have  $LM_{k,U}(\phi_n)(V) \longrightarrow 0$. 
Let $\chi : \mathcal{H}^k\rightarrow [0,1]$ be a continuous function taking value 1 on $K$ and vanishing outside $V$. We have $\langle LM_{k,U}(\phi_n), \chi \rangle \longrightarrow 0$, so that $\mu(K)\leq \langle\mu, \chi \rangle =0$. The result then follows from (\ref{eq:InnerReg}).
\end{proof}
 
\begin{corolaire}\label{cor:Intrinsic}
Let $(\phi_n)$ satisfy (\ref{eq:DefEigen}), and let $\nu\in \sigma^{\iota_k}_{k,U}((\phi_n)_n)$. Then $\nu\in \mathcal{M}$, and $\nu$ is independent of $k$ and $\iota_k$. In other words, if $F\in \mathcal{C}_b(FP)$, $k_1,k_2\in \N$ and $\iota_{k_1}, \iota_{k_2}$ satisfy (\ref{eq:ExtensionProp}), we have that $\langle LM_{k_1,U}^{\iota_{k_1}}, F\rangle$ converges if and only if $\langle LM_{k_2,U}^{\iota_{k_2}}, F \rangle$ converges, and if this is the case, their limits are then equal.
\end{corolaire} 
\begin{proof}
Let $\nu\in \sigma^{\iota_k}_{k,U}((\phi_n)_n)$. By Corollary \ref{cor:Compactness}, there exists $\mu \in \sigma_{k,U}((\phi_n)_n)$ such that for all $F\in \mathcal{C}(FP)$, $\langle \nu, F \rangle = \langle  \mu, \iota_k F \rangle$.

Since $FP$ is closed, we may define a measure $\mu$ on $FP$ as the restriction of the measure $\mu$. Let $F\in \mathcal{C}_b(FP)$. We have 
\begin{align*}
\langle \nu,F \rangle &= \int_{\mathcal{H}^k} \iota_k F \mathrm{d}\mu \\
&= \int_{FP} \iota_k F \mathrm{d}\mu + \int_{\mathcal{H}^k\backslash FP} \iota_k F \mathrm{d}\mu\\
&= \int_{FP} F \mathrm{d}\mu \\
&= \langle \mu,F \rangle.
\end{align*}

Therefore, $\nu= \mu$, so that $\nu\in \mathcal{M}$ and $\nu$ does not depend on $\iota_k$.

 Let us show that $\nu$ does not depend on $k$.
Let $F\in \mathcal{C}_b(FP)$, $k_1,k_2\in \N$ and $\iota_{k_1}, \iota_{k_2}$ satisfy (\ref{eq:ExtensionProp}). Suppose that $\langle LM_{k_1,U}^{\iota_{k_1}}, F \rangle$ converges.
Suppose first that $k_2\geq k_1$. Then, since $\mathcal{C}_b(\mathcal{H}^{k_1})\subset \mathcal{C}_b(\mathcal{H}^{k_2})$,  $\iota_{k_1}$ can be seen as a map from $\mathcal{C}(FP)$ to $\mathcal{C}_b(\mathcal{H}^{k_2})$ satisfying (\ref{eq:ExtensionProp}). Since the convergence of $\langle LM_{k_1,U}^{\iota_{k_2}}, F \rangle$ does not depend on the choice of the map $\iota_{k_2}:\mathcal{C}(FP) \rightarrow  \mathcal{C}_b(\mathcal{H}^{k_2})$ satisfying (\ref{eq:ExtensionProp}), we deduce that $\langle LM_{k_2,U}^{\iota_{k_2}}, F \rangle$ converges.

Suppose now that $k_2\leq k_1$. Then, since $\mathcal{C}_b(\mathcal{H}^{k_2})\subset \mathcal{C}_b(\mathcal{H}^{k_1})$, $\iota_{k_2}$ can be seen as a map from $\mathcal{C}(FP)$ to $\mathcal{C}_b(\mathcal{H}^{k_1})$ satisfying (\ref{eq:ExtensionProp}). Since the convergence of $\langle LM_{k_1,U}^{\iota_{k_1}}, F \rangle$ does not depend on the choice of the map $\iota_{k_1}:\mathcal{C}(FP) \rightarrow  \mathcal{C}_b(\mathcal{H}^{k_1})$ satisfying (\ref{eq:ExtensionProp}),  we deduce that $\langle LM_{k_2,U}^{\iota_{k_2}}, F \rangle$ converges.
\end{proof}

\begin{definition}
If the conclusions of Corollary \ref{cor:Intrinsic} are satisfied, we will say that $\nu$ is \emph{a local weak limit of $(\phi_n)_n$ on $U$}, and write $\nu\in \sigma_U((\phi_n)_n)$. If $\sigma_U((\phi_n)_n) = \{\nu\}$, we will say that $\nu$ is \emph{the local weak limit of $(\phi_n)_n$ on $U$}.
\end{definition}

We now give examples of sequences of eigenfunctions whose local weak limits can easily be computed.

\subsection{Two simple examples}\label{subsec:Examples}
\subsubsection*{A plane wave on the torus}
Let $\xi_1 = (1,0,...,0)\in \R^d$.
On $\mathbb{T}^d$, consider $\phi_n(x):= \cos(n x\cdot \xi_1)$. We have $-\Delta \phi_n = n^2 \phi_n$, and, if $x\in \mathbb{T}^d$ and $y=(y_1,...,y_d)\in \R^d$, $|y|<1$, we have
$$\phi_{x,n}(y) = \cos (n x\cdot \xi_1 + y\cdot \xi_1).$$
 
For  $\theta\in [0, 2\pi)$, we shall write $f_\theta(y):= \cos (\theta + y\cdot \xi_1)$.
 
\begin{lemme}
For any Borel set $U\subset \T^d$  of positive Lebesgue measure, we have $\sigma_U((\phi_n)_n)= \{\mu_{\xi_1}\}$, where
$$\mu_{\xi_1}:= \frac{1}{2\pi}\int_0^{2\pi} \delta_{f_\theta} \mathrm{d}\theta.$$
\end{lemme}
\begin{proof}
Let $U\subset \mathbb{T}^d$, let $k\in \N$ and let $F \in \mathcal{C}_b(\mathcal{H}^k)$. We have
\begin{equation*}
\begin{aligned}
\langle LM_{k,U},F\rangle &= \frac{1}{\mathrm{Vol}(U)}\int_U F(f_{nx\cdot \xi_1}) \mathrm{d}x\\
&= \frac{1}{\mathrm{Vol}(U)}\int_U F_{\xi_1}(nx) \mathrm{d}x,
\end{aligned}
\end{equation*}
where $F_{\xi_1}(x) := F(f_{x\cdot \xi_1})$ is a continuous function of $x\in \mathbb{T}^d$. Fix $(x_2,...,x_d)\in \mathbb{T}^{d-1}$, and write $U_{x_2,...,x_d}:= \{x_1; (x_1,...,x_d)\in U\}$. By expanding $F_{\xi_1}(\cdot, x_2,...,x_d)$ in Fourier modes and using Lebesgue's lemma, it is straightforward to check that 
\begin{equation*}
\begin{aligned}
\int_{U_{x_2,...,x_d}} F_{\xi_1}(nx) \mathrm{d}x_1 \longrightarrow \mathrm{Vol}(U_{x_2,...,x_d}) \int_{\mathbb{T}^1} F_{\xi_1}(x_1) \mathrm{d} x_1.
\end{aligned}
\end{equation*}
Integrating over $(x_2,...,x_d)$, the lemma follows.
\end{proof}

\subsubsection*{Concentrating eigenfunctions}
Consider $(\phi_n)_n$ a normalized sequence of eigenfunctions of the Dirichlet Laplacian in a domain $\Omega\subset \R^d$, which concentrates on a set of volume zero in the sense that
\begin{equation}\label{eq:Concentration}
\forall \varepsilon>0, \exists U_\varepsilon \subset \Omega \text{ open, with } \Vol(\Omega\backslash U_\varepsilon)< \varepsilon \text{ and } \lim\limits_{n\rightarrow \infty} \int_{U_\varepsilon} |\phi_n|^2(x) \mathrm{d}x = 0.
\end{equation}

Examples of sequences of eigenfunctions satisfying (\ref{eq:Concentration}) can easily be constructed in the disk (the so called "whispering gallery modes"): see \cite[\S 1.3]{anantharaman2016wigner} and the references given therein.

\begin{proposition}\label{prop:Concentration}
If $(\phi_n)_n$ satisfies (\ref{eq:Concentration}), then for any Borel set $U\subset \Omega$ of positive Lebesgue measure , we have $\sigma_U((\phi_n)_n) = \{\delta_0\}$.
\end{proposition}
\begin{proof}
Let $U\subset \Omega$, let $F\in \mathcal{C}(\mathcal{H}^0)$ and let $\varepsilon>0$. By continuity, there exists $\eta>0$ such that $\forall g\in \mathcal{K}$, $\boldsymbol{d}_0(g,0)<\eta \Longrightarrow  |F(g)-F(0)|<\varepsilon.$

By the assumption we made and Markov's inequality, we have, for all $\eta_1>0$, 
\begin{equation}\label{eq:OftenSmall}
\lim\limits_{n\rightarrow \infty} \Vol \big{(} \{x\in U; |\phi_n(x)|>\eta_1\}\big{)} = 0.
\end{equation}

We claim that
\begin{equation}\label{eq:ClosetoZero}
 \lim\limits_{n\rightarrow \infty} \Vol \big{(} \{x_0\in U; \boldsymbol{d}_0(\tilde{\phi}_{n,x_0},0)>\eta\}\big{)} = 0.
 \end{equation}

Indeed, suppose for contradiction that (\ref{eq:ClosetoZero}) does not hold, and that there exists $\varepsilon_1>0$ such that, if we write $\mathcal{B}_n:=\{x_0\in U; \boldsymbol{d}_0(\tilde{\phi}_{n,x_0},0)>\eta\}$, then for all $n\in \N$,  $\Vol \big{(} \mathcal{B}_n\big{)} \geq \varepsilon_1$.

Recall that, if $x_0\in \mathcal{B}_n$, then there exists $x\in B(x_0, \eta^{-1})$ such that $|\phi_n(x_0 + h_n x)| >\eta $. 
By the proof of Lemma \ref{lem:Tight}, we may find $C=C(\eta, \varepsilon_1)$ and a Borel set $U_{\varepsilon_1,\eta}$ such that $\Vol(U\backslash U_{\varepsilon_1,\eta})< \frac{\varepsilon_1}{2}$, and, for all $x_0\in U_{\varepsilon_1,\eta}$ and all $n\in \N$,
 we have $\|\tilde{\phi}_{n, x_0}\|_{C^1(B(0,\eta^{-1}))}<C$.

Therefore, if $x_0\in \mathcal{B}_n\cap U_{\varepsilon_1, \eta}$, then $|\phi_n(y)|>\frac{\eta}{2}$ for all $y$ in some ball of radius $\frac{\eta h_n}{2C}$ contained in $B(x_0, \frac{h_n}{\eta})$.

Now, since $\Vol \left(\mathcal{B}_n\cap U_{\varepsilon_1, \eta}\right)\geq \frac{\varepsilon_1}{2}$, we may find $C'\varepsilon_1 \left(\frac{h_n}{\eta}\right)^d$ disjoint balls $B(x_j, \frac{\eta}{2Ch_n})$ with $x_j\in \mathcal{B}_n\cap U_{\varepsilon_1, \eta}$. Therefore, $\{y\in U; |\phi_n(x)> \frac{\eta}{2}\}$ has a volume bounded from below independently of $n$, thus contradicting (\ref{eq:OftenSmall}). Hence, we deduce (\ref{eq:ClosetoZero}).

Therefore, we obtain that for all $\varepsilon>0$, we may find $n_\varepsilon\in \N$ such that for all $n\geq n_\varepsilon$,
\begin{equation*}
\Vol \big{(} \{x_0\in U \text{ such that } |F(\tilde{\phi}_{x_0,n})-F(0)| >\varepsilon\}\big{)}<3\varepsilon.
 \end{equation*}

We deduce from this that
\begin{equation*}
\langle LM_{0,U}(\phi_n), F \rangle \longrightarrow F(0),
\end{equation*}
which gives the result.
\end{proof}

\section{Criteria of convergence}\label{sec:Criteria}
The following lemma gives a useful criterion to determine the weak-* limit of the sequence $LM_{k,U}(\phi_n)$. It seems classical, but we could not find a proof of it in the literature.

Recall that we say that $\mathcal{A}\subset \mathcal{C}_b(X)$ is a unitary sub-algebra which separates points if it is an algebra containing the constant functions, such that for all $x_1\neq x_2\in X$, there exists $a\in \mathcal{A}$ such that $a(x_1)\neq a(x_2)$.

\begin{lemme}\label{lem:Criterion}
Let $X$ be a Polish space, and let $\mathcal{K}\subset (\mathcal{C}_b(X))^*$ be compact for the weak-* topology. Let $(\mu_n)$ be a sequence of elements of $\mathcal{K}$, and suppose that all accumulation points of $\mu_n$ for the weak-* topology are finite Borel measures on $X$.  Let $\mu$ be a finite Borel measure on $X$.

Suppose that there exists $\mathcal{A}\subset \mathcal{C}_b(X)$ a unitary sub-algebra which separates points, such that we have
\begin{equation*}
\forall a\in \mathcal{A}, \langle \mu_n, a\rangle \longrightarrow \langle \mu, a \rangle.
\end{equation*}
Then $\mu$ is the limit of $\mu_n$ in the weak-* topology.
\end{lemme}
\begin{proof}
By assumption, we know that there exists a finite Borel measure $\nu$ on $X$ and a subsequence $n_j$ such that for all $f\in \mathcal{C}_b(X)$, we have 
\begin{equation*}
\langle \mu_{n_j}, f\rangle \longrightarrow \langle \nu, f \rangle.
\end{equation*}
We want to show that $\nu=\mu$. 

Suppose for contradiction that $\mu\neq \nu$, so that there exists $f\in \mathcal{C}_b(X)$ with $\|f\|_{\mathcal{C}_b(X)}=1$ such that $\langle \mu, f \rangle \neq \langle \nu, f \rangle$.
 Since $\mu$ and $\nu$ are regular, we may find a compact set $K\subset X$ such that 
\begin{align*}
\mu(X\backslash K) &< \frac{|\langle \mu, f \rangle - \langle \nu, f \rangle|}{10}\\
\nu(X\backslash K) &< \frac{|\langle \mu, f \rangle - \langle \nu, f \rangle|}{10}.
\end{align*}

 By the Stone-Weierstrass Theorem, we may find $a\in \mathcal{A}$ such that 
$$\sup\limits_{x\in K}|a(x)- f(x)|< \frac{|\langle \mu, f \rangle - \langle \nu, f \rangle|}{10}.$$

We have
\begin{align*}
|\langle \mu, f \rangle - \langle \nu, f \rangle|&= \Big{|}\int_K f \mathrm{d}\mu - \int_K f \mathrm{d}\nu + \int_{X\backslash K} f \mathrm{d}\mu - \int_{X\backslash K} f \mathrm{d}\nu\Big{|}\\
&\leq \Big{|}\int_K f \mathrm{d}\mu - \int_K f \mathrm{d}\nu \Big{|} + 2 \frac{|\langle \mu, f \rangle - \langle \nu, f \rangle|}{10}\\
&\leq \Big{|}\int_K (f-a) \mathrm{d}\mu - \int_K (f-a)\mathrm{d}\nu \Big{|} + \Big{|}\int_K a \mathrm{d}\mu - \int_K a \mathrm{d}\nu \Big{|} + \frac{|\langle \mu, f \rangle - \langle \nu, f \rangle|}{5}\\
&\leq\Big{|}\int_K a \mathrm{d}\mu - \int_K a \mathrm{d}\nu \Big{|} + 2\frac{|\langle \mu, f \rangle - \langle \nu, f \rangle|}{5}.
\end{align*}

This is absurd, since by assumption, we have $\langle \mu, a\rangle = \langle \nu,a \rangle$. The lemma follows.
\end{proof}

Let us describe one application of Lemma \ref{lem:Criterion}. Another one will be given in the next subsection.

Let $\ell\in \N$, If $\boldsymbol{\psi}=(\psi_1,...,\psi_\ell)\in \left(C_c^\infty(\R)\right)^\ell$, and $\boldsymbol{x}=(x_1,...,x_\ell)\in (\R^d)^\ell$, we set
$$F_{\boldsymbol{\psi},\boldsymbol{x},\ell} : \mathcal{H}^k \ni f \mapsto \prod_{j=1}^\ell \psi(f(x_j)),$$
which defines an element of $\mathcal{H}^0$.

The set of all linear combinations of functions $F_{\boldsymbol{\psi},\boldsymbol{x},\ell}$ and of constant functions forms an algebra, and it separates points: if $f_1 \neq f_2 \in \mathcal{H}^0$, there exists $x$ such that $f_1(x)\neq f_2(x)$. Therefore, if $\psi\in C_c^\infty(\R)$ is such that $\psi(f_1(x))\neq \psi(f_2(x))$, we have $F_{\psi, x, 1}(f_1)\neq F_{\psi, x, 1}(f_2)$.

 We deduce the following corollary, which says that, to find the local weak limit of a sequence of eigenfunctions, we just have to test the local weak measures against functionals involving finitely many values of the rescaled functions.

\begin{corolaire}
Let $(\phi_n)$ be as in (\ref{eq:DefEigen}).
Suppose that there exists a probability measure $\mu\in \mathcal{M}^0$ such that for all $\ell\in \N$, all $\boldsymbol{\psi}\in \big{(}C_c^\infty(\R)\big{)}^\ell$ and all $\boldsymbol{x}\in (\R^d)^\ell$, we have 
\begin{equation*}
\langle LM_{0,U}(\phi_n) ,F_{\boldsymbol{\psi},\boldsymbol{x},\ell} \rangle \longrightarrow \langle \mu, F_{\boldsymbol{\psi},\boldsymbol{x},\ell}\rangle.
\end{equation*}

Then $\sigma_{U}((\phi_n)_n)= \{\mu\}$.
\end{corolaire}

We shall now give another characterization of local weak limits, in terms of local Fourier coefficients. For simplicity, we only do the construction in dimension 2, but an analogous construction could be done in any dimension.

\subsection*{Local Fourier coefficients in dimension 2}
Let $\Phi\in C^\infty(\R^2; \R)$ satisfy $-\Delta \Phi = \Phi$. The function $\Phi$ may then be written in polar coordinates as
\begin{equation}\label{eq:EigenFormula}
\Phi(r, \theta) = \frac{A_0}{2}J_0(r) + \sum_{m\geq 1} J_{m}(r)\left[A_m(\Phi) \cos(m\theta) + B_m(\Phi) \sin (m\theta) \right],
\end{equation}
where $J_m$ is the Bessel function of the first kind of order $m$. The coefficients $A_m, B_m$ may be recovered by Fourier inversion as follows: for any $r_0>0$, we have
\begin{equation}\label{eq:defCoef}
\begin{aligned}
J_{m}(r_0) A_m(\Phi) &= \frac{1}{\pi } \int_{\Sp^1} \Phi(r_0,\theta) \cos(m\theta) \mathrm{d}\theta\\
J_{m}(r_0) B_m(\Phi) &= \frac{1}{\pi } \int_{\Sp^1} \Phi(r_0,\theta) \sin(m\theta) \mathrm{d}\theta.
\end{aligned}
\end{equation}

When $\Phi(r,\theta)= e^{i r \cos(\theta-\theta_0)}$, i.e., when $\Phi$ is a plane wave, a standard computation shows that
\begin{equation}\label{eq:CoefOP}
A_m(\Phi)= 2\cos (m\theta_0), ~~~~~~B_m(\Phi) = -2\sin  (m \theta_0).
\end{equation}

If $N\in \N$, and $\Phi\in FP$, we shall write 
$$\beta_N(\Phi):= (A_0(\Phi),..., A_N(\Phi),B_1(\Phi),...,B_N(\Phi))\in \R^{2N+1}.$$

Note that, if $\boldsymbol{d_0}(\Phi_1,\Phi_2)\leq r_0$, then $\sup_{B(0,r_0)} |\Phi_1-\Phi_2|\leq \boldsymbol{d_0}(\Phi_1,\Phi_2)$. Therefore, for any $N\in \N$, we may choose $r_N<1$ such that $J_m(r_N)\neq 0$ for all $m\in \{0,...,N\}$. We deduce from this that, for any $N\in \N$, we may find $C(N)>0$ such that 
\begin{equation}\label{eq:DistZero}
\forall \Phi_1,\Phi_2\in FP, \boldsymbol{d}_0(\Phi_1,\Phi_2)\leq 1 \Longrightarrow ~~ |\beta_N (\Phi_1) - \beta_N (\Phi_2)| \leq C(N) \boldsymbol{d}_0(\Phi_1,\Phi_2). 
\end{equation}

\begin{lemme}\label{lem:criteria}
Let $\phi_n$ be a sequence of Laplace eigenfunctions, and let $\nu\in \mathcal{M}$. Suppose that for all $k\in \N$ and $\iota_k$ as in (\ref{eq:ExtensionProp}), and all $N\in \N$, we have 
\begin{equation}\label{eq:SeqConverge2}
\forall F\in C_c^\infty(\C^{2N+1}), ~~ \langle LM_{k,U}^{\iota_k}, F\circ \beta_N \rangle \longrightarrow \langle  \nu,  F\circ\beta_N \rangle.
\end{equation}

Then $\sigma_U((\phi_n)_n)=\{\nu\}$.
\end{lemme}
\begin{proof}
The lemma follows from Lemma \ref{lem:Criterion} and from the fact that the functions $F\circ \beta_N $, where $N\in \N$ and $F\in C_c^\infty(\R^{2N+1})$ form an algebra which separates points.
\end{proof}

\begin{remarque}\label{rem:CasBerry}
To show that  $\sigma_U((\phi_n)_n)=\{\nu\}$, it is therefore sufficient to show that $(\beta_N)_* LM_{k,U}^{\iota_k}$ converges weakly to $(\beta_N)_*\nu$. In the special case when $\nu=\mu_{Berry}$, its push-forward by $\beta_N$ is the normal vector in $\R^{2N+1}$ with covariance matrix $\mathrm{diag}(4, 2,..., 2)$.
\end{remarque}

\section{Local weak limits of toral eigenfunctions}\label{sec:Torus}
In this section, we will take $\Omega=\mathbb{T}^2=\R^2 / \Z^2$.
If $E\in \N$, we set
\begin{equation*}
\begin{aligned}
\mathcal{E}_E := \{\xi\in \Z^2; |\xi|^2 = E\}, ~~~~
N_E:= |\mathcal{E}_E|.
\end{aligned}
\end{equation*}

Non-trivial solution to the equation
\begin{equation*}
-\Delta \varphi_E = 4\pi^2 E \varphi_E
\end{equation*}
exist if and only if $N_E\neq 0$, and can be put in the form
\begin{equation}\label{eq:EigenTorus}
\varphi_E(x) = \sum_{\xi\in \mathcal{E}_E} a_\xi e^{2i\pi x\cdot \xi},
\end{equation}
with $a_\xi\in \C$.

If $\varphi_E$ is of the form (\ref{eq:EigenTorus}), we set
$$\mu_{\varphi_E}:= \sum_{\xi\in \mathcal{E}_E} |a_\xi|^2 \delta_{\xi/\sqrt{E}},$$
which is a measure on $\Sp^1$.

To state our theorem, we will need some assumptions on the number of arithmetic cancellations in the set $\mathcal{E}_E$.

\begin{definition}
(1) We say that a set of distinct
\begin{equation*}
\xi_1,...,\xi_\ell \in \mathcal{E}_E
\end{equation*}
is minimally vanishing if
\begin{equation}\label{eq:LinearRelation}
\xi_1+...+\xi_\ell= 0
\end{equation}
and no proper sub-sum of (\ref{eq:LinearRelation}) vanishes.

(2) We say that $E$ satisfies the condition $I(\gamma,B)$ for $0 < \gamma < \frac{1}{2}$
, $B\geq 1$ if for all $3 \leq \ell \leq  B$, the number of minimally vanishing subsets of $\mathcal{E}_E$ of length $\ell$ is at most $N_E\gamma \ell $.
\end{definition}

\begin{hyp}\label{hyp:tore}
We will suppose that the family of energies $E=E_n$, and the family of eigenfunctions $f_E$ satisfy the following conditions.

\begin{enumerate}
\item $N_E\rightarrow \infty$.

\item $a_{-\xi}= \overline{a_\xi}$.

\item $\sum_{\xi\in \mathcal{E}_E} |a_\xi|^2 = 1$.

\item There exists a function $g : \R^+\rightarrow \R^+$ satisfying for all $\delta>0$, $g(x) = o_{x\rightarrow \infty} (x^\delta)$ such that
\begin{equation*}
\max_{\xi\in \mathcal{E}_E} \{|a_\xi|^2\}\leq \frac{g(N_E)}{N_E}.
\end{equation*}

\item $\mu_{f_E}\overset{\ast}{\rightharpoonup} \mathrm{Leb}_{\Sp^1}.$

\item There exists $B(E)$ such that $\lim\limits_{n\rightarrow \infty} B(E_n) = +\infty$ and  $E$ satisfies $I(\gamma,B(E))$.
\end{enumerate}
\end{hyp}

Under these assumptions, we have:

\begin{theoreme}\label{th:Tore}
Suppose that $(E_n)$, $\varphi_{E_n}$ satisfy Hypothesis \ref{hyp:tore}. Then for any Borel set $U\subset \T^2$ of positive Lebesgue measure, we have
$$\sigma_U((\varphi_{E_n})_n) = \{\mu_{Berry}\}.$$
\end{theoreme}

\subsubsection*{Discussion of the assumptions}
Point 1. is clearly a necessary condition of the theorem. 

Points 2. and 3. of Hypothesis \ref{hyp:tore} imply that we consider families eigenfunctions which are real and normalized. We made these assumptions to follow \cite{buckley2015number}, but our results easily generalize to eigenfunctions which are not real, or which have a different normalization.

It is not clear how much point 4. could be relaxed, but we certainly have to make a hypothesis of this kind: an extreme case where 4. is violated is when two coefficients $a_\xi, a_{-\xi}$ are equal to $\frac{1}{\sqrt{2}}$, while all the other vanish: we are then in the framework of Section \ref{subsec:Examples}, where the local weak limit obtained is not $\mu_{Berry}$.

Concerning point 5., we may always assume, up to extracting a subsequence that 
$\mu_{\varphi_E}\overset{\ast}{\rightharpoonup} \rho$ for some probability measure $\rho$ on $\mathbb{S}^1$. If the measure is not $\mathrm{Leb}_{\Sp^1}$, the theorem would still hold with a similar proof, but the local weak limit we would obtain would be anisotropic. Since $\mu_{\varphi_E}\overset{\ast}{\rightharpoonup} \mathrm{Leb}_{\mathbb{S}^1}$ for a density one sequence of $E$ (see for instance \cite[Proposition 6]{fainsilber2006lattice}), we chose to present only this simplest case. 

Thanks to \cite[Theorem 17]{bombieri2014problem}, Point 6. holds for a density 1 sequence of $E$. Therefore, Theorem \ref{th:Tore} implies Theorem \ref{th:Tore1}.

Before proving Theorem \ref{th:Tore}, let us recall some of the elements of the proofs of  Bourgain, Buckley and Wigman.

\subsubsection*{Reminder on the constructions and results of Bourgain, Buckley and Wigman}
Fix $K>1$ large, $\delta>0$ and set for $-K+1\leq k \leq K$
$$I_k:= \left\{e^{i\theta} ; \theta \in \Big{(} \frac{\pi(k-1)}{K}, \frac{\pi k}{K}\Big{]}\right\},$$
which we identify with a subset of the unit circle in $\R^2$.
We also set 
\begin{align*}
\mathcal{E}_E^k&:= \{ \xi \in \mathcal{E}_E; \frac{\xi}{\sqrt{E}}\in I_k\}\\
\zeta^{(k)}&:= e^{i \theta^{(k)}}, ~~\text{where } \theta^{(k)}:=  \frac{(2k-1)\pi}{K},
\end{align*}
so that $\zeta^{(k)}$ is the middle of $I_k$.

If $x\in \T^2$, $y\in \R^2$, we set
\begin{align*}
\tilde{\varphi}_{x,E}(y) &:=\varphi_E\Big{(}x+ \frac{y}{2\pi\sqrt{E}}\Big{)} = \sum_{\xi\in \mathcal{E}_E} a_\xi e^{2i\pi x\cdot\xi} e^{i y\cdot \frac{\xi}{\sqrt{E}}}\\ 
\psi_{x,E}(y)&:= \frac{1}{\sqrt{2K}}\sum_{k=-K+1}^K  \alpha_k(x) e^{i \zeta^{(k)}\cdot y},
\end{align*}
where
\begin{align*}
\alpha_k(x) :=  \sum_{\xi\in \mathcal{E}_E^k} \sqrt{2K}  a_\xi e^{2i\pi x\cdot \xi}.
\end{align*}

To sum up informally, we first rescaled our eigenfunctions at wavelength around a point $x$, which gave us a sum of plane waves. We then regrouped these plane waves into $\frac{1}{2K}$ packets according to their directions of propagations, and sum the coefficients over each packet. The following lemma, which follows from the proofs of  Lemma 5.1 and Lemma 5.2 in \cite{buckley2015number}, tells us that we don't lose much information by doing so.

\begin{lemme}\label{lem:GoodApprox}
Let $R, \varepsilon>0$. There exists $\delta>0, K>0$ and $E_0$ such that for all $E\geq E_0$, we have
\begin{equation}\label{eq:GoodApprox}
\int_{\T^2} \|\tilde{\varphi}_{x,E}-\psi_{x,E}\|_{C^0(B(0,R))} \mathrm{d}x <\varepsilon.
\end{equation}
\end{lemme}

Now, when varying $x$, the amplitude $a_\xi e^{2i\pi x\cdot \xi}$ will vary, and we can think of $x$ as a random parameter, and of $a_\xi e^{2i\pi x\cdot \xi}$ as a random variable. If all of these random variables were independent, then the Central Limit Theorem would tell us that each $\alpha_k(x)$ behaves like a Gaussian variable.

Of course, the $a_\xi e^{2i\pi x\cdot \xi}$ are not independent variables, but the fact that there are very few arithmetic cancellations implies that this is not far from being the case, so that the $\alpha_k(x)$ have statistics which are close to being Gaussian. This is the essence of Bourgain's derandomization trick, which we now state rigorously.

Let us set
\begin{equation*}
\Psi_{\omega}^K(y):= \frac{1}{\sqrt{2K}} \sum_{k=-K+1}^K  c_k(\omega) e^{i \zeta^{(k)}\cdot y},
\end{equation*}
where $(c_k)_{k=1,..., K}$ is a sequence of iid $N_\C(0, 1)$ random variables, defined on a probability space $(\Xi, \mathbb{P})$, and where $c_{k}= \overline{c_{k+K}}$ for $k= -K+1,...,0$. 

The following lemma follows from \cite[Proposition 3.3]{buckley2015number}.
\begin{lemme}\label{lem:derandomization}
Let $R, \varepsilon >0$, $\delta<1<K$. There exists $E_0$ such that, for all $E\geq E_0$, we may find $\Xi_n'\subset \Xi$ with $\mathbb{P}(\Xi\backslash \Xi'_n)<\varepsilon$ and a measure-preserving map\footnote{In the sense that, for any measurable set $A\subset \Xi'_n$, $\tau_n(A)$ is a Borel set, and $\mathrm{Leb}(\tau_n(A)) = \mathbb{P}(A)$.} $\tau_n : \Xi_n'\rightarrow \T^2$ such that for all $\omega\in \Xi_n'$, we have
\begin{equation*}
\| \Psi^K_{\omega} - \psi_{\tau_n(\omega), E}\|_{C^0(B(0,R))}< \varepsilon.
\end{equation*}
\end{lemme}

Thanks to equation (\ref{eq:CoefOP}), we have for every $m\in \Z$
\begin{equation}\label{eq:CoefsRandom}
\begin{aligned}
A_m (\Psi_{\omega}^K) &=\sqrt{\frac{2}{K}}  \sum_{k=-K+1}^K c_k(\omega) \cos\left(m \theta^{(k)}\right) = \frac{2}{\sqrt{K}}  \sum_{k=1}^K c'_k(\omega)  \cos\left(m \theta^{(k)}\right) \\
B_m (\Psi_{\omega}^K) &= - \sqrt{\frac{2}{K}} \sum_{k=-K+1}^K c_k(\omega) \sin\left(m \theta^{(k)}\right) =-\frac{2}{\sqrt{K}}  \sum_{k=1}^K c''_k(\omega)  \sin\left(m \theta^{(k)}\right),
\end{aligned}
\end{equation}
where the variables $(c'_k), (c''_k)$ are independent real normal Gaussian variables, respectively corresponding to $\sqrt{2}$ times the real and imaginary part of $c_n$.

\subsubsection*{Proof of Theorem \ref{th:Tore}}
\begin{proof}
Let $N\in \N$. We write $X_{k,K}(\omega) := \begin{pmatrix} X^{0}_{k,K}(\omega) \\ ...\\ X^{N}_{k,K}(\omega) \end{pmatrix}$, with $X^{m}_{k,K}(\omega):= 2 c'_k(\omega) \cos\left(m \theta^{(k)}\right)$, which implicitly depends on $K$ through the definition of $\zeta^{(k)}$. The vectors $X_{k,K}$ are therefore independent, identically distributed, with zero average.

We have $\mathrm{Covar}\left[X_{k,K}^m X_{k,K}^{m'}\right]=4 \cos\left(m \theta^{(k)}\right)\cos\left(m' \theta^{(k)}\right)$, so that
$$\lim\limits_{K\to +\infty} \frac{1}{K}\sum_{k=1}^K \mathrm{Covar}\left[X_{k,K}^m X_{k,K}^{m'}\right] = \begin{cases}
4 \text{ if } m=m'=0\\
2  \text{ if } m=m'\neq 0\\
0  \text{ if } m\neq m'.
\end{cases}.$$

We are thus in position to apply the multivariate Lindeberg-Feller Central Limit Theorem, to deduce that $\left(A_m (\Psi_{\omega}^K)\right)_{m=0}^N = \left(\frac{1}{\sqrt{K}} \sum_{k=1}^K X_{k,K}\right)_{m=0}^N$ converges in distribution to the Gaussian vector in $\R^{N+1}$ with covariance matrix  $\mathrm{diag}(4, 2,..., 2)$. The same argument shows that the vectors $\left(B_m (\Psi_{\omega}^K)\right)_{m=1}^N$ converge to a Gaussian vector in $\R^{N}$ with covariance matrix $\mathrm{diag}(2,..., 2)$, independent from the one corresponding to the $A_m$.

Let $F\in C_c^\infty(\R^{2N+1})$. From what precedes (and Remark \ref{rem:CasBerry}), we deduce that for any $\varepsilon>0$,  we may find $K$ such that 
\begin{equation}\label{eq:TCL}
|\E\big{[} F( \beta_N(\Psi^K_{\omega})\big{]} - (\beta_N)_*\mu_{Berry}(F)|<\varepsilon.
\end{equation}

We now want to use (\ref{eq:TCL}) to show (\ref{eq:SeqConverge2}).
Equation (\ref{eq:GoodApprox}), combined with (\ref{eq:DistZero}), implies that for all $N\in \N$, $\varepsilon>0$, by possibly taking $K$ larger, we may find $E(N,\varepsilon)>0$ such that for all $E\geq E(N,\varepsilon)$, we have
\begin{equation*}
\int_{U} |\beta_N (\psi_{x,E})- \beta_N(\tilde{\varphi}_{x,E})|\mathrm{d}x <\varepsilon^2 \Vol(U).
\end{equation*}

In particular, 
$$\frac{1}{\Vol(U)} \Vol\{x\in U;  |\beta_N (\psi_{x,E})-  \beta_N(\tilde{\varphi}_{x,E})| > \varepsilon\} < \varepsilon,$$
so that 
\begin{equation}\label{eq:a}
\frac{1}{\Vol(U)} \Vol\{x\in U;  |F(\beta_N (\psi_{x,E}))-  F(\beta_N(\tilde{\varphi}_{x,E}))| > \varepsilon\} < C\varepsilon,
\end{equation}
for some $C$ depending only on $F$.

By Lemma \ref{lem:derandomization}, we have for $E$ large enough
\begin{equation}\label{eq:b}
\Big{|}\frac{1}{\Vol(U)} \int_U F(\beta_N (\psi_{x,E})) \mathrm{d}x - \E \big{[} F(\beta_N (\Psi_{\omega}^K))\big{]}\Big{|}< C'\varepsilon.
\end{equation}

Hence, combining (\ref{eq:TCL}), (\ref{eq:a}) and (\ref{eq:b}), we have that for all $E$ large enough,
$$\Big{|}\frac{1}{\Vol(U)} \int_U F(\beta_N(\tilde{\varphi}_{x,E}))) -  ( \beta_N)_*\mu_{Berry}(F)\Big{|}<C'' \varepsilon.$$
We may hence apply Lemma \ref{lem:criteria} to conclude.
\end{proof}

\section{Consequences of Berry's conjecture}\label{sec:Consequences}
In this section, we present some consequences of our interpretation of Berry's conjecture.
\subsection{Berry's conjecture at smaller scales}
First of all, we show that, if Berry's conjecture holds for a family of eigenfunctions restricted to a Borel set $U$, then it holds for these eigenfunctions restricted to all Borel sets $U'\subset U$ of positive measure.
\begin{proposition}\label{prop:BerryPlusPetit}
Let $U\subset \Omega$ be a Borel set of positive Lebesgue measure, and suppose that $\sigma_U((\phi_n)_n)= \{\mu_{Berry}\}$.  Then for any Borel set $U'\subset U$  of positive Lebesgue measure, we have $\sigma_{U'}((\phi_n)_n)= \{\mu_{Berry}\}$.
\end{proposition}
\begin{corolaire}\label{Cor:Equi}
Suppose that $(\phi_n)_n$ is such that $\|\phi_n\|_{L^2}^2= \mathrm{Vol}(\Omega)$ for all $n$, and   $\sigma_\Omega((\phi_n)_n)= \{\mu_{Berry}\}$. Then for any $U\subset \Omega$, we have
$$\int_{U} |\phi_n|^2(x) \mathrm{d}x \longrightarrow \Vol(U).$$
\end{corolaire}
\begin{proof}[Proof of the corollary]
If $\Vol(U)=\Vol(\Omega)$, then the result is trivial. We will thus suppose that $0<\Vol(U)<\Vol(\Omega)$.
Up to extracting a subsequence, we may suppose that $\int_{U} |\phi_n|^2(x) \mathrm{d}x \longrightarrow c$ and $\int_{\Omega\backslash U} |\phi_n|^2(x) \mathrm{d}x \longrightarrow c'$ for some $c, c'\in [0,1]$. We must of course have $c+c'=\Vol(\Omega)$.

Consider the functionals $F,F_N\in \mathcal{C}_b(\mathcal{H}^0)$ given by $F(f)= |f(0)|^2$, $F_N= \max (F, N)$.

We know that $\langle F_N,\mu_{Berry} \rangle= 1 + o_{N\to +\infty}(1)$. Since $F\geq F_N$, we have\footnote{The quantity $\langle LM_U(\phi_n), F \rangle\in [0, +\infty]$ is well defined, since $F$ is positive, and $LM_U(\phi_n)$ is a positive measure.}  
$$\frac{1}{\Vol(U)}\int_U |\phi_n|^2(x) \mathrm{d}x = \langle LM_U(\phi_n), F \rangle \geq  \langle LM_U(\phi_n), F_N \rangle,$$
 so that $c \geq \Vol(U)$. Similarly, we must have $c'\geq \Vol(\Omega\backslash U)$. Therefore, these inequalities must be equalities, and $c=\Vol(U)$.
\end{proof}

Before proving the proposition, let us introduce some notations.

Let $y\in \R^d$. If $f\in \mathcal{H}^k$, we define $\tau_y f\in \mathcal{H}^k$ by $\big{(}\tau_y f\big{)}(\cdot) = f(y+\cdot)$.  If $F\in \mathcal{C}_k$, we define $\tau_y F\in \mathcal{C}_k$ by $\big{(}\tau_y F\big{)}(f) = F \big{(}\tau_y f\big{)}$ for all $f\in \mathcal{H}^k$. Finally, if $\mu\in \mathcal{M}^k$, we define $\tau_y \mu\in \mathcal{M}^k$ by $\big{\langle}\tau_y\mu, F \big{\rangle}  = \big{\langle} \mu, \tau_y F\big{\rangle}$ for all $F\in \mathcal{C}_k$.

\begin{lemme}
Let $\mu\in \sigma_U((\phi_n)_n)$. Then, for any $y\in \R^d$, $\tau_y\mu=\mu$.
\end{lemme}
\begin{proof}[Proof of the Lemma]
Up to extracting a subsequence, we may suppose that $\sigma_U((\phi_n)_n)=\{\mu\}$.
Let $F\in \mathcal{C}_k$. We have
\begin{align*}
\left\langle LM_{k,U}(\phi_n), \tau_y F\right\rangle &= \frac{1}{\Vol(U)} \int_U \mathrm{d}x_0 F(\tau_y\tilde{\phi}_{x_0,n})\\
&= \frac{1}{\Vol(U)} \int_U \mathrm{d}x_0 F(\tilde{\phi}_{x_0+ h_n y,n})\\
&= \frac{1}{\Vol(U)} \int_{U-h_n y} \mathrm{d}x_1 F(\tilde{\phi}_{x_1,n})\\
&= \frac{1}{\Vol(U)} \int_{U} \mathrm{d}x_1 F(\tilde{\phi}_{x_1,n}) + O\left(\Vol \left(U \Delta(U-h_n y)\right)\right)\\
&= \left(LM_{k,U}(\phi_n)\right)(F) + o_{n\to \infty}.
\end{align*}

Here, $U \Delta(U-h_n y)$ is the symmetric difference between $U$ and $U-h_n y$, and its volume goes to zero with $n$, since\footnote{When $U$ is an open set, the fact that $\Vol(U\cap (U-\varepsilon)) \underset{\varepsilon\to 0}{\longrightarrow} 0$ follows from the dominated convergence theorem. Using regularity of the Lebesgue measure, we deduce the same statement when $U$ is a general Borel set.} $\Vol(U\cap (U-h_n y))\underset{n\to \infty}{\longrightarrow} 0$.

Taking the limit $n\to +\infty$, we obtain $\mu(\tau_y F) = \mu(F)$, so that $\tau_y\mu = \mu$.
\end{proof}
\begin{proof}[Proof of Proposition \ref{prop:BerryPlusPetit}]
If $\Vol(U')=\Vol(U)$, then we have $LM_{U'}(\phi_n) = LM_{U}(\phi_n)$, so the result is trivial. We may this suppose that $0<\Vol(U')<\Vol(U)$.

By definition, we have $LM_{k,U}(\phi_n) = \frac{\Vol(U')}{\Vol(U)} LM_{k,U'}(\phi_n)  + \frac{\Vol(U\backslash U')}{\Vol(U)} LM_{k,U\backslash U'}(\phi_n)$. Up to extracting a subsequence, we may suppose that  $LM_{k,U'}(\phi_n)$ and $LM_{k,U\backslash U'}(\phi_n)$ have weak limits $\mu_{U'}$ and $\mu_{U\backslash U'}$ respectively. We then deduce that
$$\mu_{Berry} =  \frac{\Vol(U')}{\Vol(U)} \mu_{U'}  + \frac{\Vol(U\backslash U')}{\Vol(U)} \mu_{U\backslash U'}.$$

By the Fomin-Grenander-Maruyama theorem\footnote{See for instance \cite[Appendix B]{nazarov2015asymptotic} for a proof of this theorem}, $\mu_{Berry}$ is ergodic for the action of the translations $(\tau_y)_{y\in \R^d}$.  Therefore, since $\mu_{U'}$ and $\mu_{U\backslash U'}$ are $\tau_y$-invariant by the previous lemma, we must have $\mu_{U'}= \mu_{U\backslash U'} = \mu_{Berry}$ as claimed.
\end{proof}

\begin{corolaire}
Suppose that $\sigma_U((\phi_n)_n)= \{\mu_{Berry}\}$. If $b : U \longrightarrow \R$ is a Borel function and $F\in \mathcal{C}_b(\mathcal{H}^k)$, we have
\begin{equation}\label{eq:SuperBerry}
\int_U b(x) F(\tilde{\phi}_{x,n}) \mathrm{d}x \longrightarrow \left( \int_U b(x) \mathrm{d}x \right) \langle \mu_{Berry}, F \rangle.
\end{equation} 
\end{corolaire}
\begin{proof}
By Proposition \ref{prop:BerryPlusPetit}, we know that (\ref{eq:SuperBerry}) holds when $b$ is a step function. By linearity, it holds for any simple function, and by density, it holds for any Borel function.

\end{proof}

\subsection{Quantum unique ergodicity}\label{sec:QUE}
Let us denote by $\mathcal{S}_{-\infty}(\Omega)$ the space of functions $a\in C^\infty(\Omega\times \R^d)$ such that, for any $N\in \N$, for any multi-index $\alpha \in \N^{2N}$, and $k\in \N$, there exists $C_{N,k}$ such that
$$\sup\limits_{x\in \Omega, \xi \in \R^d}\left|\partial^\alpha a(x,\xi)\right| \leq C_{N,k} \langle \xi \rangle^{-N}.$$

Let $a\in\mathcal{S}_{-\infty}(\Omega)$. We define the standard quantization of $a$, depending on a small parameter $h>0$ as an operator $\mathrm{Op}_h(a) : L^2(\Omega)\longrightarrow L^2(\Omega)$ by
$$\left( \mathrm{Op}_h(a) u\right)(x) = \frac{1}{(2\pi h)^d}\int_{\R^d} \mathrm{d}\xi \int_\Omega \mathrm{d}y e^{\frac{i}{h} (x-y)\cdot\xi} a(x,\xi) u(y).$$

We say that a sequence of functions $\phi_n$ satisfying (\ref{eq:DefEigen}) is \emph{quantum uniquely ergodic} or \emph{satisfies quantum unique ergodicity} if we have, for any $a\in \mathcal{S}_{-\infty}(\Omega)$, that
\begin{equation}\label{eq:QE}
\lim\limits_{n\to +\infty} \frac{\left\langle \phi_n, \mathrm{Op}_{h_n}(a) \phi_n \right\rangle}{\|\phi_n\|_{L^2}^2} = \mathrm{Area}(\mathbb{S}^{d-1}) \times \int_{S^*\Omega} a(x,\xi) \mathrm{d}\mu_{Liou}(x,\xi),
\end{equation}
where $S^*\Omega= \Omega \times \{\xi\in \R^d; \|\xi\|=1\}$, and $\mu_{Liou}$ is the Liouville probability measure, i.e., the uniform probability measure on $S^*\Omega$.

We refer the reader to \cite{Zworski_2012} for more information on the standard quantization, and on quantum unique ergodicity.
Recall that the \emph{quantum unique ergodicity conjecture} says that, if $\Omega$ is a chaotic billiard, and if $(\phi_n)$ is an orthonormal sequence of eigenfunctions of the Dirichlet Laplacian in $\Omega$, then $(\phi_n)$ is quantum uniquely ergodic.\footnote{Note that the quantum ergodicity theorem (initially proved on compact manifolds without boundary in \cite{Shn}, \cite{Zel} \cite{CdV}, and extended to biliards in \cite{gerard1993ergodic}, \cite{zelditch1996ergodicity}) says that if the billiard $\Omega$ is \emph{ergodic}, then there exists a density one subsequence $(n_k)$ such that $(\phi_{n_k})$ is quantum uniquely ergodic. The quantum unique ergodicity conjecture was stated in \cite{rudnick1994behaviour} for eigenfunctions on manifolds with negative sectional curvature, but it is natural to extend this conjecture to manifolds with an Anosov geodesic flow, and to chaotic billiards, in the sense of \cite{chernov2006chaotic}.} 

We will now show that, if a sequence of functions $(\phi_n)$ is such that $\sigma_{\Omega}((\phi_n)_n)= \{\mu_{Berry}\}$, then $\phi_n$ is quantum uniquely ergodic. In particular, our version of Berry's conjecture implies the quantum unique ergodicity conjecture.

\begin{proposition}
Let  $(\phi_n)$ be a sequence of functions such that for all $n$, $\|\phi_n\|_{L^2}^2= \Vol(\Omega)$ and $\sigma_{\Omega}((\phi_n)_n)= \{\mu_{Berry}\}$. Then $(\phi_n)$ is quantum uniquely ergodic.
\end{proposition}

\begin{proof}
Thanks to \cite[Theorem 5.1]{Zworski_2012}, it is enough to prove  (\ref{eq:QE}) for $a$ in a dense subset of $\mathcal{S}_{-\infty}(\Omega)$. By linearity, we may thus restrict ourselves to the case when $a(x,\xi) = b(x) c(\xi)$, where $b\in C_c^\infty(\Omega)$, and $\widehat{c}$ is compactly supported. We have
\begin{align*}
\left\langle \phi_n, \mathrm{Op}_{h_n}(a) \phi_n \right\rangle &=\frac{1}{(2\pi h_n)^d}\int_\Omega \mathrm{d} x \int_{\R^d} \mathrm{d}\xi \int_\Omega \mathrm{d}y e^{\frac{i}{h_n} (y-x)\cdot\xi} b(x)c(\xi) \phi_n(x) \phi_n(y).
\end{align*}
Set $y= x + h_nz$. We obtain
\begin{align*}
\left\langle \phi_n, \mathrm{Op}_{h_n}(a) \phi_n \right\rangle &=\frac{1}{(2\pi)^d}\int_\Omega \mathrm{d} x \int_{\R^d} \mathrm{d}\xi \int_{h_n^{-1}(\Omega-x)} \mathrm{d}z e^{i z\cdot\xi} b(x)c(\xi) \phi_n(x) \phi_n(x+ h_n z)\\
&= \int_\Omega \mathrm{d} x b(x) \phi_n(x) \int_{\R^d} \mathrm{d}z  \widehat{c}(-z)  \tilde{\phi}_{x,n}(z) + o(1)\\
&= \int_\Omega \mathrm{d} x b(x) F(\tilde{\phi}_{x,n})+ o(1),
\end{align*}
where $F(\Phi)= \overline{\Phi(0)} \int_{\R^d}  \widehat{c}(-z)  \Phi(z)\mathrm{d}z $. Since $\hat{c}$ is rapidly decaying, $F$ is a continuous functional on $\mathcal{H}^0$, but it is not bounded. This is why we define $F_M(\Phi) := \begin{cases}
F(\Phi) ~~\text{ if } |F(\Phi)|\leq M\\
M ~~\text{ otherwise,}
\end{cases}$ which is bounded.  We will prove the following lemma at the end of the section.

\begin{lemme}\label{lem:BoundFM}
We have
$$\lim_{M\to +\infty} \limsup_{n\to +\infty} \int_\Omega \mathrm{d}x |b(x)| \left|F(\tilde{\phi}_{x,n}) -  F_M(\tilde{\phi}_{x,n})\right| = 0.$$
\end{lemme}

By (\ref{eq:SuperBerry}), we have for any $M>0$ that 
$$\int_\Omega \mathrm{d}x b(x) F_M(\tilde{\phi}_{x,n})\longrightarrow \langle \mu_{Berry}, F_M \rangle = \langle \mu_{Berry}, F \rangle + o_{M\to +\infty}(1),$$
since $F$ is integrable with respect to $\mu_{Berry}$. Let us now compute 
\begin{align*}
\langle \mu_{Berry}, F \rangle&= \int_{\mathbb{R}^{d}} \widehat{c}(-z) \mathbb{E}\left[ \Psi_{Berry}(0) \Psi_{Berry}(z)\right]\mathrm{d}z\\
&=  \int_{\mathbb{R}^{d}}\mathrm{d}z\int_{\R^d}\mathrm{d}\xi e^{i\xi\cdot z} c(\xi) \int_{\Sp^{d-1}} e^{-i z\cdot \theta} \mathrm{d}\theta\\
&= \int_{\mathbb{S}^{d-1}} c(\theta) \mathrm{d}\theta,
\end{align*}
since $\int_{\R^d} e^{i\xi\cdot z} e^{-iz\cdot \theta} \mathrm{d}z = \delta_{\xi=\theta}$.

We deduce from this and Lemma \ref{lem:BoundFM} that
$$\left\langle \phi_n, \mathrm{Op}_{h_n}(bc) \phi_n \right\rangle \longrightarrow \int_{\Omega} b(x) \mathrm{d}x \int_{\mathbb{S}^{d-1}} c(\xi) \mathrm{d}\xi.$$
The result follows.
\end{proof}
\begin{proof}[Proof of Lemma \ref{lem:BoundFM}]
We know that $\widehat{c}$ is supported in a ball $B(0,R)$ for some $R>0$, so that $|F(\Phi)|\leq C \sup_{y\in B(0,R)} |\Phi(y)|^2$. By the Sobolev embeddings, we know that this quantity is smaller than a constant times $\|\Phi\|_{H^k(B(0,R+1))}$ for some $k\in \N$. Therefore, we have
$$|F(\Phi)|\leq C \int_{B(0,R+1)} \left|\sum_{j=0}^k \Delta^j \Phi(y)\right|^2 \mathrm{d}y.$$

When $\Phi = \tilde{\phi}_{x,n}$, we know that, for $y\in B(0,R+1)$, we have $\Delta^j\Phi (y) = (-1)^j \Phi(y)$, unless $x$ is at a distance $o_{h_n\to 0}(1)$ from the boundary of $\Omega$. Since $b\in C_c^\infty(\Omega)$, this does not happen when $n$ is large enough. Therefore, we have for $n$ large enough,
$$\forall x \in \mathrm{supp}(b), |F(\tilde{\phi}_{x,n})|\leq C\int_{B(0,R+1)} |\tilde{\phi}_{x,n}(y)|^2 \mathrm{d}y.$$

We have
\begin{equation}\label{eq:BoundFG}
\begin{aligned}
\int_\Omega \mathrm{d}x |b(x)| \left|F(\tilde{\phi}_{x,n}) -  F_M(\tilde{\phi}_{x,n})\right| &\leq \int_\Omega \mathrm{d}x |b(x)| \left|2 F(\tilde{\phi}_{x,n})\right| \mathbf{1}_{\left|F(\tilde{\phi}_{x,n})\right|\geq M}\\
&\leq C\int_\Omega \mathrm{d}x |b(x)| G(\tilde{\phi}_{x,n})  \mathbf{1}_{G(\tilde{\phi}_{x,n}) \geq M/C},
\end{aligned}
\end{equation}
where $G(\Phi) := \int_{B(0,R+1)} |\Phi(y)|^2 \mathrm{d}y$.

Let us write, for $N\in \N$, $G_N:= \max (G, N)$, which belongs to $\mathcal{C}^0$.
We have
\begin{align*}
\int_\Omega |b(x)| G(\tilde{\phi}_{x,n}) \mathrm{d}x &= \int_{B(0,R+1)} \mathrm{d}y \int_\Omega  \mathrm{d}x |b(x)| |\phi_n|^2(x+h_ny)\\
&=\int_{B(0,R+1)} \mathrm{d}y \int_\Omega  \mathrm{d}x |b(x-h_n y)| |\phi_n|^2(x)\\
&= \Vol(B(0,R+1)) \int_\Omega |b(x)| |\phi_n(x)|^2 \mathrm{d}x + o_{n\to +\infty}(1)\\
&= \Vol(B(0,R+1)) \int_\Omega |b(x)| \mathrm{d}x + o_{n\to +\infty}(1),
\end{align*}
as can be easily deduced from Corollary \ref{Cor:Equi}.

On the other hand, we have thanks to (\ref{eq:SuperBerry}) that 
\begin{align*}
\int_\Omega |b(x)| G_N(\tilde{\phi}_{x,n}) \mathrm{d}x & \underset{n\to +\infty}{\longrightarrow} \left( \int_U |b(x)| \mathrm{d}x \right) \langle \mu_{Berry}, G_N \rangle\\
& = \left( \int_U |b(x)| \mathrm{d}x \right) \langle \mu_{Berry}, G \rangle + o_{N\to +\infty}(1)\\
&=  \left( \int_U |b(x)| \mathrm{d}x \right)\Vol(B(0,R+1)) + o_{N\to +\infty}(1).
\end{align*}

Therefore, we deduce that
\begin{equation*}
\limsup\limits_{n\to +\infty} \int_\Omega |b(x)| \big{(}G-G_N\big{)} (\tilde{\phi}_{x,n}) \mathrm{d}x = o_{N\to +\infty}(1),
\end{equation*}
so that
\begin{equation*}
\limsup_{n\to +\infty}\int_\Omega |b(x)| G(\tilde{\phi}_{x,n})  \mathbf{1}_{G(\tilde{\phi}_{x,n}) \geq M/C}  \mathrm{d}x=o_{M\to +\infty}(1).
\end{equation*}
Combining this and (\ref{eq:BoundFG}), we deduce that
$$\lim_{M\to +\infty} \limsup_{n\to +\infty} \int_\Omega \mathrm{d}x |b(x)| \left|F(\tilde{\phi}_{x,n}) -  F_M(\tilde{\phi}_{x,n})\right| = 0,$$
as claimed.
\end{proof}

\subsection{Application to nodal domains counting}\label{sec:Nodal}
If $\phi\in C(\Omega)$, we shall write $\mathcal{N}(\phi)$ for the number of nodal components of $\phi$, i.e.
\begin{equation*}
\mathcal{N}(\phi) = \sharp \left\{ \text{connected components of } \Omega\backslash \phi^{-1}(0)\right\}.
\end{equation*}

If $r>0$, we shall denote by $\mathcal{N}_{r}(\phi)$ for the number of nodal domains whose diameter is smaller than $r$.

If $\Phi\in FP$, or, more generally, if $\Phi \in C(\R^d;\R)$, the nodal domains of $\Phi$ are the connected components of $\{x\in \R^d ; \Phi(x)\neq 0\}$. If $x\in \R^d$ and $r>r'>0$, we shall denote by $N(x,r,r',\Phi)$ the number of nodal domains of $\Phi$ included in $B(x,r)$, and whose diameter is smaller than $r'$. We will also write $N(r,r',\Phi):= N(0,r,r',\Phi)$, and $N(r,\Phi):= N(0,r,r,\Phi)$ for the number of nodal domains of $\Phi$ included in $B(0,r)$.

It was shown in \cite{nazarov2015asymptotic} that the map $FP \ni \Phi\mapsto N(r,\Phi)$ belongs to $L^1(\mu_{Berry})$, and that 
\begin{equation*}
\langle \mu_{Berry}, N(r,\cdot)\rangle \underset{r\to +\infty}{\sim} c_{NS}  r^d,
\end{equation*} where $c_{NS}$ is a positive constant, called the Nazarov-Sodin or the Bogomolny-Schmit constant.

Actually, the arguments of \cite{nazarov2015asymptotic} show that we have 
\begin{equation}\label{eq:NS}
\frac{1}{r^d} \langle \mu_{Berry}, N(r,r',\cdot)\rangle \underset{r\to +\infty}{\longrightarrow} c_{NS}(r'),
\end{equation}
where $c_{NS}(r')\underset{r' \to +\infty}{\longrightarrow} c_{NS}$.

\begin{proposition}\label{prop:BerryNodal}
Let $(\phi_n)$ be an orthonormal sequence of eigenfunctions in $\Omega$ such that $\sigma_\Omega((\phi_n)_n) = \{\mu_{Berry}\}$. Then we have
$$ h_n^{d}\mathcal{N}_{R h_n} (\phi_n) \underset{n\to +\infty}{\longrightarrow}  c_{NS}(R).$$
\end{proposition}

Before proving the proposition, let us state two corollaries.

\begin{corolaire}
Let $(\phi_n)$ be an orthonormal sequence of eigenfunctions in $\Omega$ such that $\sigma_\Omega((\phi_n)_n) = \{\mu_{Berry}\}$. Then we have
$$ \liminf\limits_{n\to +\infty} h_n^{d}\mathcal{N} (\phi_n) \geq c_{NS}.$$
\end{corolaire}

\begin{proof}
For any $R>0$, we have $\mathcal{N}(\phi_n)\geq \mathcal{N}_{Rh_n}(\phi_n)$, so that
$$\liminf\limits_{n\to +\infty} h_n^d\mathcal{N}(\phi_n)\geq c_{NS}(R).$$
Taking the limit as $R$ goes to $+\infty$ gives us the result.
\end{proof}

One would expect that if $\sigma_\Omega((\phi_n)_n) = \{\mu_{Berry}\}$, we actually have $ \lim\limits_{n\to +\infty} h_n^{d}\mathcal{N} (\phi_n) = c_{NS}$. This is not easy to show, since by definition, local weak limits only allow us to count nodal domains of diameter $O(h_n)$. There could be nodal domains which are much larger, and we have no bound on them.

However, we can estimate the number of large nodal domains, and hence, have an upper bound on $\mathcal{N}(\phi_n)$, if we work \emph{in dimension 2}, and we have a bound on the \emph{nodal length} of $\phi_n$.

The nodal volume of $\phi_n$ is defined as 
$$NV(\phi_n):= \mathrm{Haus}_{d-1} \left(\{x\in \Omega ; \phi_n(x)=0\}\right),$$
and we refer the reader to \cite{HanLin} for a proof that $\{x\in \Omega ; \phi_n(x)=0\}$ has a well-defined $(d-1)$-dimensional Hausdorff measure.

It was conjectured by Yau in \cite{Yau} that the following bound holds in any  dimension, in any domain (and on any manifold), for any sequence of eigenfunctions of the Laplacian:

\begin{equation}\label{eq:Yau}
\exists C>0,~ \forall n\in \N, ~ NV(\phi_n)\leq \frac{C}{h_n}.
\end{equation}

This bound is known to hold when $\Omega$ is an analytic domain (or an analytic manifold), as was shown in \cite{DF} (see also \cite{HanLin} for a self-contained proof). Although some recent breakthroughs were made in \cite{logunovSurf} \cite{logunov2018nodal}, this bound was not established for general smooth domains or manifolds.

\begin{corolaire}
Suppose that $d=2$, and that (\ref{eq:Yau}) holds. Let $(\phi_n)$ be an orthonormal sequence of eigenfunctions in $\Omega$ such that $\sigma_\Omega((\phi_n)_n) = \{\mu_{Berry}\}$, and such that (\ref{eq:Yau}) holds. Then we have
$$ \lim\limits_{n\to +\infty} h_n^{2}\mathcal{N} (\phi_n) = c_{NS}.$$
\end{corolaire}
\begin{proof}
We have to estimate $\mathcal{N}(\phi_n)- \mathcal{N}_{Rh_n}(\phi_n)$, hence, to count the number of nodal domains whose diameter is larger than $R h_n$.

 The boundary of each such nodal domain will have a length of at least\footnote{This is not true in higher dimension, which is why we restrict ourselves to dimension 2. Indeed, in higher dimension, we can find connected sets with volume 1, diameter going to infinity, but with a boundary whose area remains bounded: for example, take two balls attached by a very thin tube.}  $2R h_n$. Therefore
$$\big{(}\mathcal{N}(\phi_n)- \mathcal{N}_{Rh_n}(\phi_n)\big{)} R h_n \leq NM(\phi_n) \leq \frac{C}{h_n}.$$
We hence get
$$\big{(}\mathcal{N}(\phi_n)- \mathcal{N}_{Rh_n}(\phi_n)\big{)} h_n^{2} \leq \frac{C}{R}.$$

Letting $R$ go to $+\infty$, this goes to zero. The result follows.
\end{proof}

We now turn to the proof of Proposition \ref{prop:BerryNodal}.

\begin{proof}
To prove the proposition, we begin by recalling some regularity properties of the map $N(r,r',\cdot)$.

\begin{lemme}
For any $\ell\in \N$, $r>r'>0$, the sets
$$A_\ell(r,r'):= \left\{ \Phi\in FP; N(r,r',\Phi)=\ell \right\}$$
are in the Borel $\sigma$-algebra of $FP$, and $\mu_{Berry}(\partial A_\ell(r,r')) = 0$.
\end{lemme}
\begin{proof}
This result was proven in \cite[Lemma 6 and \S 6.2.2]{nazarov2015asymptotic} using Bulinskaya's lemma. The key point is to note that a function $\Phi$ belongs to $\partial A_\ell$ if arbitrarily small $C^1$ perturbations of $\Phi$ will change the number of nodal domains of $\Phi$ included in $B(0,r)$ of radius $<r'$, which by the implicit function theorem, implies that there exists either
\begin{itemize}
\item  $x_0\in B(0,r)$ such that $\Phi(x_0)=0$ and $(\nabla \Phi)(x_0)=0$;
\item  or a nodal domain of $\Phi$ which is included in $\overline{B(0,r)}$, but not in $B(0,r)$;
\item or  a nodal domain of $\Phi$ which is included in $\overline{B(x,r')}$, but not in $B(x,r')$ for some $x\in B(0,r)$.
\end{itemize} 
All of these events happen with probability zero.
\end{proof}

We may therefore use the Portmanteau theorem to deduce that for any $\ell\in \N$, $r>r'>0$, we have
\begin{equation}\label{eq:portmanteau}
\left(LM_{\Omega}(\phi_n)\right)(A_\ell(r,r'))  \underset{n\to +\infty}{\longrightarrow} \mu_{Berry}(A_\ell(r,r')).
\end{equation}

The Faber-Krahn inequality implies that there exists a constant $c>0$ such that if $\Phi\in FP$, then the nodal domains of $\Phi$ must have a volume larger than $c$.
This implies that for each $r>0$, the map $FP\ni \Phi\mapsto N(r,\Phi)$ is bounded by some constant $C(r)$, since we can pack only finitely many domains of volume larger than $c$ in $B(0,r)$.

We deduce from this that $N(r,r',\Phi)= \sum_{\ell=0}^{C(r)} \ell \mathbf{1}_{\Phi\in A_\ell}$, so that, by (\ref{eq:portmanteau}), we have
\begin{equation}\label{eq:portmanteau2}
\left\langle LM_{\Omega}(\phi_n), N(r,\cdot) \right\rangle \underset{n\to +\infty}{\longrightarrow} \left\langle \mu_{Berry}, N(r,\cdot)\right\rangle.
\end{equation}

Next, we recall the ``sandwich estimate" of \cite{nazarov2015asymptotic}, which make the link between counting small nodal domains locally and globally.

\begin{lemme}\label{lem:Sandwich}
Let $r,r'>0$ and $\phi \in C(\R^d;\R)$ vanish outside of $\Omega$.
We have
$$\frac{1}{\Vol(B(0,R))}  \int_{\Omega} N(x,r,r', \phi)\mathrm{d}x \leq \mathcal{N}_{r'}(\phi)\leq \frac{1}{\Vol(B(0,R+R'))}  \int_{\Omega} N(x,r+r',r', \phi) \mathrm{d}x.$$
\end{lemme} 
\begin{proof}
The proof of this lemma is the same as that of \cite[Lemma 1]{nazarov2015asymptotic}, but we recall it for the reader's convenience.

Let us denote by $\mathcal{O}_{r'}$ the set of nodal domains of $\phi$ whose diameter is smaller than $r'$. If $x\in \Omega$ and $O\in \mathcal{O}_{r'}$, we have
$$O\subset B(x,r) \Longleftrightarrow x\in \bigcap_{y\in O} B(y,r).$$
Therefore,
\begin{align*}
\int_\Omega N(x,r,r',\phi) \mathrm{d}x &= \int_\Omega \left( \sum_{O\in \mathcal{O}_{r'}} \mathbf{1}_{O\subset B(x,r)} \right) \mathrm{d}x\\
&=  \sum_{O\in \mathcal{O}_{r'}} \int_\Omega  \mathbf{1}_{x\in \bigcap_{y\in O} B(y,r)} \mathrm{d}x\\
& = \sum_{O\in \mathcal{O}_{r'}} \Vol \Big{(}\Big{\{} x\in \bigcap_{y\in O} B(y,r)\Big{\}}\Big{)}.
\end{align*}

We always have $\Vol \Big{(}\Big{\{} x\in \bigcap_{y\in O} B(y,r)\Big{\}}\Big{)}\leq \Vol(B(0,r))$. Since $O$ has diameter smaller than $r'$, we also have the converse inequality $\Vol \Big{(}\Big{\{} x\in \bigcap_{y\in O} B(y,r)\Big{\}}\Big{)}\geq \Vol(B(0,r-r'))$. We therefore deduce that
$$\mathcal{N}_{r'}(\phi) \Vol(B(0,r-r')) \leq \int_\Omega N(x,r,r',\phi) \mathrm{d}x \leq \mathcal{N}_{r'}(\phi) \Vol(B(0,r)),$$
from which the result follows.
\end{proof}

We may apply the previous lemma to $\phi=\phi_n$ continued by zero outside of $\Omega$, $r= R h_n$, $r'=R'h_n$ with $R>R'>1$. Recall that when $\varepsilon>0$, we write $\Omega_\varepsilon := \{x\in \Omega; \mathrm{d}(x,\partial \Omega)<\varepsilon\}$. We have 
\begin{align*}
\int_{\Omega} N(x,R h_n, R'h_n, \phi_n) \mathrm{d}x &= \int_{\Omega\backslash \Omega_{h_n}} N(x,R h_n, R'h_n, \phi_n) \mathrm{d}x + o_{n\to +\infty}(1) \\
&= \int_{\Omega\backslash \Omega_{h_n}} N(R,R', \tilde{\phi}_{x,n}) \mathrm{d}x + o_{n\to +\infty}(1)\\
& = \left\langle LM_\Omega(\phi_n), N(R,R', \cdot) \right\rangle + o_{n\to +\infty}(1),
\end{align*}
where we used twice the fact that $N(R,R',\cdot)$ is bounded, and $\Vol(\Omega\backslash \Omega_{h_n}) = o_{n\to +\infty}(1)$.

Combining this, Lemma \ref{lem:Sandwich} and (\ref{eq:portmanteau2}), we deduce that for any $R>1$, we have
\begin{align*}
\frac{1}{\Vol(B(0,R))}\left\langle \mu_{Berry}, N(R, R', \cdot) \right\rangle &\leq  \liminf\limits_{n\to +\infty} h_n^d \mathcal{N}_{R'}(\phi_n) \\
&\leq \limsup\limits_{n\to +\infty} h_n^d \mathcal{N}_{R'}(\phi_n)  \\
&\leq  \frac{1}{\Vol(B(0,R+R'))} \left\langle \mu_{Berry}, N(R+R', R', \cdot) \right\rangle.
\end{align*}

The proof of (\ref{eq:NS}) shows that, when we let $R$ goes to $+\infty$, the left and right-hand side of the previous inequality both converge to $c_{NS}(R')$, which proves the result.
\end{proof}

\subsubsection*{Lower bounds on the number of nodal domains for more general local weak limits}

Courant's nodal theorem implies that if $\phi_n,$ and $h_n$ are as in (\ref{eq:DefEigen}), we have 
$$\mathcal{N}(\phi_n)= O(h_n^{-d}).$$

If $\Omega$ is the square, we can find some examples of eigenfunctions with arbitrarily large eigenvalue, but only two nodal domains (see \cite{stern1925bemerkungen}, or, more recently,  \cite{berard2015dirichlet}).
Therefore, no non-trivial lower bound exist for $\mathcal{N}(\phi_n)$, and few examples are known where one can prove that $\mathcal{N}(\phi_n)\longrightarrow \infty$. 
We shall now see how local weak limits can be useful if one hopes to show that some family of eigenfunctions satisfies $\mathcal{N}(\phi_n)\longrightarrow \infty$.

A general method for finding lower bounds on the number of nodal domains is the \emph{barrier method}, introduced in \cite{nazarov2009number}, which consists in finding small nodal domains that are stable under small perturbations. We illustrate here how this general idea can be combined with local weak limits to find lower bounds on the number of nodal domains. Namely, the following proposition tells us that if a local weak limit puts some mass on a small neighbourhood of a function with a stable nodal domain, then the whole sequence of eigenfunctions will have many nodal domains.

\begin{definition}
Let $\eta>0$ and let $f\in \mathcal{H}^0$. We will say that $f$ has an $\eta$-stable nodal domain if
 $f(0)> \eta$ and 
 $\exists U\subset B(0, \eta^{-1})$ open and connected, with $0\in U$, and $\forall x \in \partial U$, $f(x)< -\eta$.
\end{definition}

\begin{proposition}
Let $\phi_n$ be a sequence of Laplace eigenfunctions with $\sigma_\Omega((\phi_n)_n) = \{\nu\}$.  Suppose that there exists $\eta>0$ and $\Phi\in FP$ having an $\eta$-stable nodal domain, such that 
$$\nu \big{(}B_{\boldsymbol{d}_0}(\Phi,\frac{\eta}{2})\big{)}>0.$$
 Then there exists $c>0$ such that
\begin{equation}\label{eq:borneinf}
\mathcal{N}(\phi_n) > c h_n^{-d}.
\end{equation}
\end{proposition}
\begin{proof}
Since $\Phi$ has an $\eta$-stable nodal domain, if $f\in C^\infty(\R^d)$ satisfies $\boldsymbol{d}_0 (\Phi, f)<\frac{\eta}{2}$, then $f$ has an $\frac{\eta}{2}$-stable nodal domain. In particular, $0$ is contained in a nodal domain of $f$ included in $B(0, \eta)$.

Now, by hypothesis, we may find $\varepsilon>0$ such that for all $n$ large enough, $\Vol (\{x_0\in \Omega, \tilde{\phi}_{x_0,n}\in B(\Phi,\eta/2)\})>\varepsilon$. Each point in this set of positive volume, is contained in a nodal domain of radius $\eta h_n$. The lower bound (\ref{eq:borneinf}) follows.
\end{proof}

\subsection{$L^\infty$ norms}\label{sec:Linfini}
\subsubsection*{A lower bound}
Let $(\phi_n)$ be an orthonormalized sequence of eigenfunctions such that $\sigma_\Omega((\phi_n)_n)= \{\mu_{Berry}\}$. Then we have
\begin{equation}\label{eq:LowerInfinity}
\|\phi_n\|_{L^\infty(\Omega)} \underset{n\to +\infty}{\longrightarrow} + \infty.
\end{equation}

Indeed, for any $M>1$, let $\chi_M\in C_c^{\infty}((M-1,M+2),[0,1])$ take value $1$ on $[M,M+1]$. Then the map $F_M (f)= \chi_M(|f(0)|)$ belongs to $\mathcal{C}^0$. Therefore, we have by assumption
$$ \left\langle LM_\Omega(\phi_n), F_M\right\rangle \underset{n\to +\infty}{\longrightarrow} \left\langle \mu_{Berry}, F_M\right\rangle>0.$$
Now, we note that $\left\langle LM_\Omega(\phi_n), F_M\right\rangle \leq \mathrm{Leb} \{x\in \Omega ; |\phi_n(x)| \in (M-1, M+2)\}$. We easily deduce from this that $\liminf\limits_{n\to +\infty} \|\phi_n\|_{\infty} \geq M-1$. Since this is true for any $M$, (\ref{eq:LowerInfinity}) follows.

\begin{remarque}
Considering the functional $F_M(\Phi):=\max \left(|\Phi(0)|^p, M\right)$, we deduce that, is $\sigma_{\Omega}((\phi_n)_n)=\{\mu_{Berry}\}$, then $$\lim\limits_{n\to +\infty} \int_\Omega \max\left( |\phi_n(x)|^p, M\right) \mathrm{d}x \underset{n\to +\infty}{\longrightarrow} \langle \mu_{Berry}, F_M \rangle,$$ 
so that $\limsup_{n\to +\infty} \int_\Omega |\phi_n(x)|^p \mathrm{d}x \geq \mathbb{E}(|\Psi_{Berry}|^p)$, where $\Psi_{Berry}$ is chosen at random according to the Berry measure on $FP$. Therefore, our interpretation of Berry's conjecture gives a lower bound on the asymptotic behaviour of $L^p$ norms of eigenfunctions, but no upper bound. This is related to the discussion in the following paragraph.
\end{remarque}

\subsubsection*{No upper bound}
If  $(\phi_n)$ is an orthogonal sequence of eigenfunctions with $\|\phi_n\|_{L^2}^2=\Vol(\Omega)$ and such that $\sigma_\Omega((\phi_n)_n)= \{\mu_{Berry}\}$, we cannot in general obtain an upper bound on $\|\phi_n\|_{L^\infty}$ which is any better than the usual H\"ormander bound. The reason for this is that the local weak convergence captures information about how the eigenfunctions look like at \emph{typical} points, while the eigenfunctions become very large at very non-typical points.

Let us give a heuristic situation where $\sigma_\Omega((\phi_n)_n)= \{\mu_{Berry}\}$, but we have no further information on $\|\phi_n\|_{L^\infty}$. Suppose that we have two normalised sequence $\phi_{1,n}$, $\phi_{2,n}$ satisfying $-h_n^2 \Delta \phi_{1,n} = \phi_{1,n}$ and $-h_n^2 \Delta \phi_{2,n} = \phi_{2,n}$. Suppose further that $\sigma_\Omega((\phi_{1,n})_n)= \{\mu_{Berry}\}$, while $\phi_{2,n}$ concentrates on a set of zero volume, in the sense of (\ref{eq:Concentration}). 

Let $r_n$ be a sequence going to zero arbitrarily slowly, and set $\phi_n:=\frac{\phi_{1,n} + r_n \phi_{2,n}}{\|\phi_{1,n} + r_n \phi_{2,n}\|_{L^2}}$. Around most points, $\phi_n$ is just a small perturbation of $\phi_{1,n}$, so an argument similar to the one in the proof of Proposition \ref{prop:Concentration} shows that $\sigma_\Omega((\phi_n)_n)= \{\mu_{Berry}\}$. However, it could be that $\|\phi_{n,2}\|_{L^\infty}$ saturates the H\"ormander bound, so that $\|\phi_n\|_{L^\infty}$ would almost saturate the H\"ormander bound.

\begin{appendix}
\section{A convenient topology on $C^k(\R^d)$}\label{sec:App}
    \numberwithin{equation}{section}
Let $d\geq 1$, $k\in \N$. If $f,g\in C^k(\R^d)$, we shall write
\begin{align*}
\boldsymbol{d}_k(f,g)&:= \inf \left\{\varepsilon>0;  \|f-g\|_{C^k(B(0, \varepsilon^{-1}))} <\varepsilon \right\}\\
&=\left[\sup \{r>0; \|f-g\|_{C^k(B(0,r))} <r^{-1} \}\right]^{-1}.
\end{align*}

\begin{proposition}
For all $k\in \N$, the space $\mathcal{H}^k:= (C^k(\R^d), \boldsymbol{d}_k)$ is a Polish space, i.e. a separable and complete metric space.
\end{proposition}
\begin{proof}
It is clear that $\boldsymbol{d}_k$ defines a distance on $C^k(\R^d)$. For each $\varepsilon >0$, the space $C^k(B(0,\varepsilon^{-1}))$ is separable, so $\mathcal{H}^k$ is separable as well. Finally, if $(f_n)$ is a Cauchy sequence in $\mathcal{H}^k$, then for each $\varepsilon>0$, it is a Cauchy sequence in $C^k(B(0,\varepsilon^{-1}))$, so that it must converge in $C^k(B(0,\varepsilon^{-1}))$. From this, we see that $(f_n)$ converges in $\mathcal{H}^k$.
\end{proof}

In some proofs, we will also need to use, for $k=-1$ or $k=-2$, the distances
\begin{equation*}
\boldsymbol{d}_k(f,g):= \inf \left\{\varepsilon>0;  \|f-g\|_{H^k(B(0, \varepsilon^{-1}))} <\varepsilon \right\}
\end{equation*}
We then have that $\mathcal{H}^k:= (H^k_{loc}(\R^d), \boldsymbol{d}_k)$ is a Polish space.

Consider a sequence of positive numbers $\boldsymbol{a} = (a_\ell)_\ell$. For any $k\in \N$, we shall write 
$$\mathcal{H}^{k}(\boldsymbol{a}) := \{f\in C^{k}(\R^d) \text{ such that } \forall \ell \in \N, \|f_{|B(0,\ell)}\|_{C^{k}(B(0,\ell))} \leq a_\ell\}.$$

\begin{lemme}\label{lem:CompactSubset}
Let $k\in \N$, and $\boldsymbol{a} \in \N^\N$. The space $\mathcal{H}^{k+1}(\boldsymbol{a})$ is relatively compact in $\mathcal{H}^k$.
\end{lemme}
\begin{proof}
Let $(f_n)\in \mathcal{H}^{k+1}(\boldsymbol{a})$. For each $\ell>0$, by the Arzel\`a-Ascoli Theorem, a ball in $C^{k+1}(B(0,\ell))$ can be compactly embedded in $C^k(B(0,\ell))$. Therefore, we may extract a subsequence of $f_n$ which converges in $C^k(B(0,\ell))$. By a diagonal extraction, we may find a subsequence of $f_n$ which converges in $C^k(B(0,\ell))$ for all $\ell>0$. Therefore, this subsequence converges in $\mathcal{H}^k$.
\end{proof}

We shall denote by $\overline{\mathcal{H}^{k+1}(\boldsymbol{a})}$ the completion of $\mathcal{H}^{k+1}(\boldsymbol{a})$ with respect to the $\mathcal{H}^k$ topology. By the preceding lemma, it forms a compact metric space.

In the sequel, if $k\in \N$, we shall denote by $\mathcal{M}^k$ the Banach space of finite signed measures on $\mathcal{H}^k$, and by $\mathcal{C}_b(\mathcal{H}^k)$ the space of \emph{bounded} real-valued continuous functions on the metric space $\mathcal{H}^k$.

\subsection*{The space of free eigenfunctions}
Recall that the space $FP$ was defined in (\ref{eq:DefFP}). 

\begin{lemme}
For each $k\in \N$, $FP$ is a closed subset of $\mathcal{H}^k$ for the topology induced by $\boldsymbol{d}_k$. 
In particular, $FP$ is a Polish space.
\end{lemme}
\begin{proof}
Note that $FP$ is included in each of the $\mathcal{H}^k$. Let us check that it forms a closed subset. Let $f_n\in FP$ converge to some $f\in \mathcal{H}^k$. We then have
$$-\Delta (f_n-f) = \Delta f + f_n.$$

We have that  $-\Delta(f_n-f)$ converges to zero in the $\mathcal{H}^{k-2}$ topology.
Indeed, for any $r>0$, $\|f_n-f\|_{C^k(B(0,r))}$ converges to zero, so, if $k\geq 2$, we have$ \|\Delta(f_n-f)\|_{C^{k-2}(B(0,r))}$ converges to zero. 
If $k=0$ or $k=1$, we just use the fact that $\|f_n-f\|_{H^k(B(0,r))}\leq C(r) \|f_n-f\|_{C^k(B(0,r))}$, so it goes to zero as $n\to +\infty$. Therefore, for each $r>0$ we have, $\|\Delta(f_n-f)\|_{H^{k-2}(B(0,r))}$ goes to zero as $n\to +\infty$, so that $\Delta(f_n-f)$ converges to zero in the $\mathcal{H}^{k-2}$ topology.

 Therefore, $f_n$ converges to $-\Delta f$ in the  $\mathcal{H}^{k-2}$ topology. By uniqueness of the limit, we must have $-\Delta f = f$.
\end{proof}

\begin{lemme}
Let $k,k'\in \mathbb{N}$. The distances $\boldsymbol{d}_k$ and $\boldsymbol{d}_{k'}$ are topologically equivalent on $FP$.
\end{lemme}
\begin{proof}
Suppose that $k'\geq k$. Let $f_n, f\in FP$. It is clear that, if $\boldsymbol{d}_{k'}(f_n,f)\longrightarrow 0$, then $\boldsymbol{d}_{k}(f_n,f)\longrightarrow 0$. 

Conversely, suppose that  $\boldsymbol{d}_{k}(f_n,f)\longrightarrow 0$. This amounts to saying that, for any $r>0$, $\|f_n-f\|_{C^k(B(0,r))}\longrightarrow 0$. In particular, $\|f_n-f\|_{H^k(B(0,r))}\longrightarrow 0$. But, using the fact that $f_n, f\in FP$, we deduce that for any $j\in \N$, $\|\Delta^j(f_n-f)\|_{H^k(B(0,r))}\longrightarrow 0$, so that $\|f_n-f\|_{H^{k+2j}(B(0,r))}\longrightarrow 0$. Since this is true for any $j\in \N$, we may use the Sobolev embeddings to deduce that $\|f_n-f\|_{C^{k'}(B(0,r))}\longrightarrow 0$.
\end{proof}

 In the sequel, $FP$ will always be equipped with the topology induced by the distances $\boldsymbol{d}_k$.

Let us write $\mathcal{C}_b(FP)$ for the space of bounded continuous functions on $FP$, equipped with the sup norm. We shall write $\mathcal{M}$ for the Banach space of finite signed measures on $FP$. We shall also write $\big{(}\mathcal{C}_b(FP)\big{)}^*$ for the topological dual of $\mathcal{C}_b(FP)$.

By Tietze's extension lemma, we may find a continuous linear map $\iota_k : \mathcal{C}_b(FP)\longrightarrow \mathcal{C}_b(\mathcal{H}^k)$ such that 
\begin{equation}\label{eq:ExtensionProp}
\begin{aligned}
\iota_k F_{|FP} &= F\\
\| \iota_k F\|_{\mathcal{C}_b(\mathcal{H}^k)} &= \|F\|_{\mathcal{C}_b(FP)}.
\end{aligned}
\end{equation}
(see for instance \cite[\S 5]{dugundji1951extension} for the fact that we may take $\iota_k$ linear and continuous.)
\end{appendix}


\begin{thebibliography}{10} 

\bibitem{abert2018eigenfunctions}
M. Abert, N. Bergeron, and  E. Le Masson, \emph{Eigenfunctions and random waves in the Benjamini-Schramm limit}, arXiv preprint arXiv:1810.05601, 2018.


\bibitem{abrahamsen1997review}
P. Abrahamsen, \emph{A review of Gaussian random fields and correlation functions}. (1997)

\bibitem{anantharaman2016wigner}
N. Anantharaman and M. Léautaud and F. Macià. \emph{Wigner measures and observability for the Schr\"odinger equation on the disk}, Inventiones mathematicae 206.2 (2016), pp. 485-599.

\bibitem{aurich1993statistical}
R. Aurich and F. Steiner \emph{Statistical properties of highly excited quantum eigenstates of a strongly chaotic system}, Physica D: Nonlinear Phenomena 64.1-3 (1993), pp. 185-214.

\bibitem{backer1998rate}
A. B\" acker, R. Schubert, and P. Stifter \emph{Rate of quantum ergodicity in Euclidean billiards}, Physical Review E 57.5 (1998), p. 5425.

\bibitem{backhausz2019almost}
 \'A Backhausz and B. Szegedy. \emph{On the almost eigenvectors of random regular graphs}, The Annals of Probability 47.3 (2019), pp. 1677-1725.

\bibitem{barnett2006asymptotic}
A. Barnett. \emph{Asymptotic rate of quantum ergodicity in chaotic Euclidean billiards}, Communications on Pure and Applied Mathematics: A Journal Issued by the Courant Institute of Mathematical Sciences 59.10 (2006), pp. 1457-1488.

\bibitem{benjamini2011recurrence}
 I. Benjamini and O. Schramm. \emph{Recurrence of distributional limits of finte planar graphs},  Electron. J. Probab. 6 (2001).

\bibitem{berard2015dirichlet}
P. Bérard and B. Hellfer. \emph{Dirichlet eigenfunctions of the square membrane: Courant's property, and A. Stern's and A. Pleijel's analyses},  Analysis and geometry. Springer. 2015, pp. 69-114.




\bibitem{berry1977regular}
M.V. Berry, \emph{Regular and irregular semiclassical wavefunctions}, J. Phys. A \textbf{10 (12)} (1977), p. 2083.

\bibitem{bogomolny2002percolation}
E. Bogomolny and C. Schmit \emph{Percolation model for nodal domains of chaotic wave functions},  Physical Review Letters 88.11 (2002), p. 114102.

\bibitem{bombieri2014problem}
 E. Bombieri and J. Bourgain. \emph{A problem on sums of two squares},  International Mathematics Research Notices 2015.11 (2014), pp. 3343-3407.


\bibitem{bourgain2014toral}
J. Bourgain, \emph{On toral eigenfunctions and the random wave model}, Isr. J. Math. \textbf{201(2)} (2014), 611-630

\bibitem{buckley2015number}
J. Buckley, I. Wigman, \emph{On the number of nodal domains of toral eigenfunctions}, Ann. Inst. Henri Poincare \textbf{17(11)} (2016), 3027-3062

\bibitem{chernov2006chaotic}
N. Chernov and R. Markarian. \emph{Chaotic billiards}, Mathematical Surveys and Monographs,
127. American Mathematical Society. 2006.

\bibitem{CdV}
Y.~Colin~de~Verdi\`ere, \emph{Ergodicit\'e et fonctions propres du laplacien}, Comm. Math. Phys. \textbf{102} (1985) 497--502.

\bibitem{DFsurfaces}
H. Donnelly and C. Fefferman. \emph{Nodal sets of eigenfunctions of the Laplacian on surfaces}, J. Amer. Math. Soc. 3 (1990), pp. 333-353.

\bibitem{DF}
 H. Donnelly and C. Fefferman. \emph{Nodal sets of eigenfunctions on Riemannian manifolds}, Inventiones mathematicae 93.1 (1988), pp. 161-183.
 
 \bibitem{dud}
R.M. Dudley. \emph{Real Analysis and Probability}, Cambridge University Press, 2002.

\bibitem{dugundji1951extension}
 J. Dugundji. \emph{An extension of Tietze's theorem}, Pacific Journal of Mathematics 1.3
(1951), pp. 353-367.

\bibitem{fainsilber2006lattice}
 L. Fainsilber, P. Kurlberg, and B. Wennberg. \emph{Lattice points on circles and discrete velocity
models for the Boltzmann equation}, SIAM journal on mathematical analysis 37.6 (2006),
pp. 1903-1922.

\bibitem{gerard1993ergodic}
P. Gérard and \'E. Leichtnam. \emph{Ergodic properties of eigenfunctions for the Dirichlet problem}, Duke Mathematical Journal 71.2 (1993), pp. 559-607.

\bibitem{HanLin}
Q. Han and F.-H. Lin. \emph{Nodal Sets of Solutions of Elliptic Differential Equations}, Book available on Han's homepage, 2013.

\bibitem{hejhal1992topography}
 D. A. Hejhal and B. N. Rackner. \emph{On the topography of Maass waveforms for $PSL(2, \mathbb{Z}$)}, Experimental Mathematics 1.4 (1992), pp. 275-305.
 
 \bibitem{hul2005investigation}
 O. Hul et al. \emph{Investigation of nodal domains in the chaotic microwave ray-splitting rough
billiard}, Physical Review E 72.6 (2005), p. 066212.

\bibitem{kuhl2007nodal}
 U. Kuhl et al. \emph{Nodal domains in open microwave systems}, Physical Review E 75.3
(2007), p. 036204.

\bibitem{logunov2018nodal}
A. Logunov. \emph{Nodal sets of Laplace eigenfunctions: polynomial upper estimates of the Hausdorff measure}, Annals of Mathematics (2018), pp. 221-239.

\bibitem{logunovSurf}
A. Logunov and E. Malinnikova. \emph{Nodal sets of Laplace eigenfunctions: estimates of the
Hausdorff measure in dimensions two and three}, In: 50 years with Hardy spaces. Springer,
2018, pp. 333-344.

\bibitem{nazarov2015asymptotic}
F. Nazarov and M. Sodin. \emph{Asymptotic laws for the spatial distribution and the number of
connected components of zero sets of Gaussian random functions}, Zh. Mat. Fiz. Anal.
Geom. 12(3) (2016), pp. 205-278.

\bibitem{nazarov2009number}
F. Nazarov and M. Sodin. \emph{On the number of nodal domains of random spherical harmonics}, American Journal of Mathematics (2009), pp. 1337-1357.

\bibitem{nazarov2010random}
F. Nazarov and M. Sodin. \emph{Random complex zeroes and random nodal lines}, Proceedings
of the International Congress of Mathematicians 2010 (ICM 2010) (In 4 Volumes) Vol. I: Plenary
Lectures and Ceremonies Vols. II-IV: Invited Lectures. World Scientic. 2010, pp. 1450-
1484.

\bibitem{nonnenmacher2013anatomy}
S. Nonnenmacher. \emph{Anatomy of quantum chaotic eigenstates} In: Chaos. Springer, 2013,
pp. 193-238.

\bibitem{rudnick1994behaviour}
Z. Rudnick and P. Sarnak. \emph{The behaviour of eigenstates of arithmetic hyperbolic manifolds}, Communications in Mathematical Physics 161.1 (1994), pp. 195-213.












%




\bibitem{sartori2020planck}
A. Sartori, \emph{Nodal domains count for toral eigenfunctions at Planck-scale}, J. Funct. Anal., 2020, p. 108663.

\bibitem{savytskyy2004experimental}
N. Savytskyy, O. Hul, and L. Sirko. \emph{Experimental investigation of nodal domains in the
chaotic microwave rough billiard}, Physical Review E 70.5 (2004), p. 056209.

\bibitem{Shn}
A.~I.~\v{S}nirel'man, \emph{Ergodic properties of eigenfunctions}, Uspehi Mat. Nauk, \textbf{29} (1974) 181--182.


\bibitem{stern1925bemerkungen}
A. Stern. \emph{Bemerkungen über asymptotisches Verhalten von Eigenwerten und Eigenfunktionen}, W. Fr. Kaestner, 1925.

\bibitem{Yau}
S.-T. Yau. \emph{Open problems in geometry}, Proc. Sympos. Pure Math. 54, Part 1 (1993),
p. 1.

\bibitem{zelditch2010recent}
S. Zelditch. \emph{Recent developments in mathematical quantum chaos}, Current developments in mathematics 2009 (2010), pp. 115-204.

\bibitem{Zel}
S.~Zelditch, \emph{Uniform distribution of eigenfunctions on compact hyperbolic surfaces}, Duke Math. J. \textbf{55} (1987) 919--941.

\bibitem{zelditch1996ergodicity}
S. Zelditch and M. Zworski. \emph{Ergodicity of eigenfunctions for ergodic billiards},  Communications in mathematical physics 175.3 (1996), pp. 673-682.

\bibitem{Zworski_2012}
M. Zworski \emph{Semiclassical analysis}, American Mathematical Society (2012)


\end{thebibliography}
\end{document}